\documentclass[a4paper, 11pt]{amsart}
\usepackage{amsthm}
\usepackage[]{amsmath}
\usepackage{amssymb}
\usepackage{enumerate}
\usepackage{tabularx}
\usepackage[]{color}
\usepackage[left=2.5cm, right=2.5cm, bottom=3cm, top=3cm]{geometry}
\usepackage[colorlinks]{hyperref}
\usepackage{tikz}
\usepackage{multirow}
\usepackage{subcaption}
\usepackage{algorithm}
\usepackage{algorithmic}
\usepackage{multirow}

\newcommand\comment[1]{}

\newtheorem{assumption}{Assumption}
\newtheorem{theorem}{Theorem}
\newtheorem{lemma}{Lemma}
\newtheorem{corollary}{Corollary}

\newtheorem{remark}{Remark}

\allowdisplaybreaks[4]

\title[Unfitted FEM for Maxwell]{An unfitted finite element method  with direct
extension stabilization for
time-harmonic Maxwell problems on smooth domains}

\author
{Fanyi Yang 
}\address{Email: yangfanyi@scu.edu.cn\\  School of Mathematics, Sichuan University, Chengdu 610064, China}
\author
{Xiaoping Xie 
}
\address{Corresponding author. Email: xpxie@scu.edu.cn \\ School of Mathematics, Sichuan University, Chengdu 610064, China}

\renewcommand{\d}[1]{\mathrm d \boldsymbol{#1}}

\newcommand{\bm}[1]{\boldsymbol{#1}}
\newcommand{\bmr}[1]{\bm{\mr{#1}}}
\newcommand{\lj}{[ \hspace{-2pt} [}
\newcommand{\rj}{] \hspace{-2pt} ]}
\newcommand{\mb}[1]{\mathbb{#1}}
\newcommand{\mc}[1]{\mathcal{#1}}
\newcommand{\mr}[1]{\mathrm{#1}}
\newcommand{\jump}[1]{\lj #1 \rj}
\newcommand{\aver}[1]{ \{#1\}  }
\newcommand{\wt}[1]{ \widetilde{ #1}}

\newcommand\aseminorm[1]{| #1 |_{\bm{{a}}}}
\newcommand\anorm[1]{\| #1 \|_{\bm{{a}}}}
\newcommand\Anorm[1]{\| #1 \|_{\bm{\mathrm{A}}}}
\newcommand\cseminorm[1]{ | #1 |_{\bm{c}}}
\newcommand\cnorm[1]{\| #1 \|_{\bm{{c}}}}
\newcommand\Cnorm[1]{\| #1 \|_{\bm{\mathrm{C}}}}

\newcommand\Wnorm[1]{\| ( #1 ) \|_{\bm{\mathrm{W}}}}
\newcommand\gnorm[1]{ | #1 |_{\bm{g}}}
\newcommand\jnorm[1]{ | #1 |_{\bm{j}}}

\def\un{\bm{\mr{n}}}
\def\curl{\ifmmode \mathrm{curl} \else \text{curl}\fi}
\def\div{\ifmmode \mathrm{div} \else \text{div}\fi}
\def\Ned{\ifmmode \text{N\'ed\'elec} \else \text{N\'ed\'elec} \fi}
\def\supp{\ifmmode \text{supp} \else \text{supp} \fi}
\def\coec{\mu_{r}^{-1}}
\def\coed{\varepsilon_r}
\def\coek{k^2}
\def\pick{\Pi_{K}}

\def\MTh{\mc{T}_h}

\def\MThG{\mc{T}_h^{\Gamma}}

\def\MFh{\mc{F}_h}

\def\MFhG{\mc{F}_h^{\Gamma}}

\def\MEhcI{\mc{E}_h^{\circ, i}}

\def\MFhI{\mc{F}_h^{i}}
\def\MFhb{\mc{F}_h^b}
\def\MFhc{\mc{F}_h^{\circ}}
\def\MFhcI{\mc{F}_h^{\circ, i}}
\def\MFhcb{\mc{F}_h^{\circ, b}}

\def\MThc{\mc{T}_h^{\circ}}

\def\Oh{\Omega_h}

\def\Ohc{\Omega_h^\circ}

\def\Qh{Q_h^m}
\def\QhC{Q_h^{m, \bmr{c}}}
\def\QhD{Q_h^{m, \bmr{d}}}
\def\Qhc{Q_h^{m, \circ}}
\def\Qhcc{Q_h^{ m, \circ, \bmr{c}}}

\def\Es{E^*}

\def\Vh{\bmr{V}_h^{r}}
\def\VhC{\bmr{V}_h^{r, \bmr{c}}}
\def\VhD{\bmr{V}_h^{r, \bmr{d}}}
\def\vh{\bmr{V}_h}
\def\Vhc{\bmr{V}_h^{r, \circ}}
\def\Vhcc{\bmr{V}_h^{r, \circ, \bmr{c}}}

\def\qh{Q_h}

\def\Oho{\Oh}

\def\Ker{\text{Ker}}

\def\Cr{C_{\mathrm{reg}}}

\date{}

\begin{document}

\begin{abstract}
  We propose an unfitted finite element method for numerically solving
  the time-harmonic Maxwell equations on a smooth domain. The model
  problem involves a Lagrangian multiplier to relax the divergence
  constraint of the vector unknown. The embedded boundary of the
  domain is allowed to cut through the background mesh arbitrarily. The unfitted scheme is based on a mixed interior penalty
  formulation, where  
  Nitsche penalty method is applied to enforce  the boundary condition
  in a weak sense, and a penalty stabilization technique is adopted
  based on a local direct extension operator to ensure the stability
  for cut elements. We prove the inf-sup stability and obtain optimal
  convergence rates under the energy norm and the $L^2$ norm for both
  the vector unknown and the Lagrangian multiplier.  Numerical
  examples in both two and three dimensions are presented to
  illustrate the accuracy of the method.

  \noindent \textbf{keywords}: unfitted finite element method; direct
  extension; time-harmonic Maxwell equation;

\end{abstract}

\maketitle

\section{Introduction}
\label{sec_introduction}
Maxwell equations describe the laws of macroscopic electric and
magnetic phenomena, and have a wide range of applications in  science
and engineering fields like plasma physics, electrodynamics, antenna
design, satellites and telecommunication. 
In this paper, we present and analyze an unfitted finite element
method for  time-harmonic Maxwell equations on smooth domains.
In the literature, many
finite element methods have been developed for solving the
time-harmonic Maxwell equations, e.g. $H(\curl)$-conforming edge
element methods \cite{Bermudez2002finite, Chen2000finite,
Chen2007adaptive, Ern2018analysis, Hiptmair2002finite,
Lu2019continuous, Monk2003finite, Nedelec1980mixed, Nedelec1986new,
Zhong2009optimal}, discontinuous Galerkin methods
\cite{Feng2014absolutely, Houston2005interior, Lu2017absolutely,
Nguyen2011hybridizable, Perugia2003local, Perugia2002stabilized}, and
nodal type finite element methods
\cite{Brenner2007locally,Duan2018mixed, Monk2003finite}.

The above mentioned methods are based on fitted meshes that 
cover the computational domain exactly. For complex geometries, it  is not a
trivial task to generate
a high quality mesh to represent the domain accurately, especially in high dimensions. For  problems with
complex geometries, the unfitted finite element method has been a
popular and successful tool because the geometry description is
decoupled from the mesh generation 
\cite{Babuska2011stable, Fries2010extended, Burman2015cutfem,
Li1998immersed, Strouboulis2000design, Gurkan2019stabilized,
Lehrenfeld2016high, Xie2020extended}.  An important advance \cite{Hansbo2002unfittedFEM} is to
combine with the Nitsche-type penalization to weakly impose the
boundary condition or the interface condition, then the domain can be
easily embedded into the unfitted mesh. This method, sometimes called  Nitsche-XFEM or CutFEM, has been extensively applied in a variety
of problems, see \cite{Bordas2017geometrically, Burman2015cutfem,
Hansbo2014cut, Burman2021unfitted,
Massing2014stabilized, Gurkan2019stabilized} and the references
therein. It is noticeable that for the penalty method, one 
difficulty is the small cuts appearing in elements that are cut by
the embedded geometry, which may adversely effect the convergence of
the numerical scheme \cite{Burman2021unfitted}.  There are two common
ways to handle this issue:  one is employing some stabilization
mechanism, such as the ghost penalty \cite{Burman2010ghost}, and
the other one is applying some cell agglomeration algorithms
\cite{Johansson2013high, Burman2021unfitted, Huang2017unfitted}. 
To our best knowledge, there is no unfitted finite element method for
solving the time-harmonic Maxwell problem.

In this contribution, we develop a finite element method for solving the
time-harmonic Maxwell problem with unfitted meshes. The proposed
method is based on a mixed interior penalty formulation with the
Nitsche penalty method to enforce the boundary condition in a weak
sense. In our method, we apply a local extension operator
\cite{Yang2021unfitted}, together
with the idea of the ghost penalty stabilization
\cite{Burman2010ghost}, to address the issue caused by the small cuts
for both curl operator and Lagrangian multiplier. The method is shown
to be stable in the sense that the constants lying in the error
estimates are independent of how the embedded geometry intersects the
mesh.  The local extension operator just directly extends the
polynomial defined on the interior neighbouring element to the cut
elements. The method is easily implemented and can achieve the
high-order accuracy.  For the mixed
formulation, we prove the inf-sup stability for the form of the
divergence constraint, following the ideas in 
\cite{Guzman2018infsup}. 
The optimal convergence
rates for both solution variables under the associated energy norm and
the $L^2$ norm are obtained.  In addition, we explore the
relationship between the wave number $k$ and the constants appearing
in the error bounds.
We confirm the theoretical predictions and
illustrate the accuracy in a series of numerical examples in both two
and three dimensions. 

The rest of this paper is as follows. In Section
\ref{sec_preliminaries}, we introduce 
the local extension operator and   the associated penalty
bilinear forms. We also give some basic properties that are
fundamental in the numerical analysis. In Section \ref{sec_scheme}, we
define the mixed formulation and introduce an auxiliary formulation
which is more suitable for analysing. In Section \ref{sec_erroraux},
we present the main theoretical results including the inf-sup
stability, the continuity and the coercivity of bilinear forms and the
final error estimates.  The numerical performance of the proposed
method is tested in Section \ref{sec_numericalresults}. Finally, we
make some concluding remarks in Section \ref{sec_conclusion}.

\section{Preliminaries} 
\label{sec_preliminaries}
Let $\Omega^* \subset \mb{R}^d(d = 2, 3)$ be a polygonal (polyhedral)
domain, and we let $\Omega \subset \Omega^*$ be an open subdomain with
a $C^2$-smooth boundary $\Gamma:= \partial \Omega$.  We denote by
$\MTh^*$ the background mesh on $\Omega^*$ into triangles
(tetrahedrons). For any element $K \in
\MTh^*$, we denote by $h_K$ the diameter of $K$ and by $\rho_K$ the
radius of the largest ball inscribed in $K$.  The mesh size $h$ is
given as $h := \max_{K \in \MTh^*} h_K$. The mesh $\MTh^*$ is assumed
to be quasi-uniform in the sense that there exists a constant $C$
independent of $h$, such that for any element $K \in \MTh^*$, there
holds $h \leq C \rho_K$. 


We introduce two sets related to the domain $\Omega$, 
\begin{displaymath}
  \MTh := \{ K \in \MTh^* \ | \  K \cap \Omega \neq \varnothing\},
  \quad \MThc := \{ K \in \MTh \ | \ K \in \Omega \}.
\end{displaymath}
Here $\MTh$ is the computational mesh, which is the minimal
subset of $\MTh^*$ that entirely covers the domain $\Omega$, and
$\MThc$ is a subset of $\MTh$ consisting of all interior elements
located inside $\Omega$. We set their corresponding domains as
\begin{displaymath}
  \Oh := \text{Int}\Big(\bigcup_{K \in \MTh}  \overline{K} \Big),
  \quad \Ohc := \text{Int}\Big(\bigcup_{K
  \in \MThc} \overline{K} \Big),
\end{displaymath}
and obviously, there holds $\Ohc \subset \Omega \subset \Oh$.  We
denote by $\MFh$ the collection of all $d - 1$ dimensional faces in
$\MTh$, and then $\MFh$ is decomposed into $\MFh = \MFhI \cup \MFhb$,
where $\MFhI$ and $\MFhb$ are the sets of interior faces and boundary
faces in $\MTh$, respectively. For any $f \in \MFh$, we define $h_f$
as the diameter of $f$. Moreover, we denote by $\MFhc$ the set of all
$d - 1$ dimensional faces in $\MThc$ and, similarly, we decompose
$\MFhc$ into $\MFhc = \MFhcI \cup \MFhcb$, where $\MFhcI$ and $\MFhcb$
consist of interior faces and boundary faces in $\MThc$, respectively.
We define $\MThG$ and $\MFhG$ as the collections of the elements and
faces that are cut by the boundary $\Gamma$, 
\begin{displaymath}
  \MThG := \{K \in \MTh \ | \  K \cap \Gamma \neq \varnothing \},
  \quad \MFhG := \{ f \in \MFh \ | \ f \cap \Gamma \neq \varnothing
  \}.
\end{displaymath}
It can readily seen that $\MThG = \MTh \backslash \MThc$ and $\MFhG =
\MFhI \backslash \MFhc$. For any element $K \in \MTh$ and any face $f \in \MFh$, we define $K^0
:= K \cap \Omega$ and $f^0 := f \cap \Omega$ as their parts inside the
domain $\Omega$. For any element $K \in \MThG$, we define $\Gamma_K =
\Gamma \cap K$.

We make the following geometrical assumptions
\cite{Hansbo2002unfittedFEM, Massing2014stabilized, Wu2012unfitted,
Gurkan2019stabilized, Guzman2018infsup}, which can be always fulfilled
for the fine enough mesh, to ensure the embedded boundary $\Gamma$ is
well-resolved.
\begin{assumption}
  For any cut face $f \in \MFhG$, $f$ is intersected by $\Gamma$ at
  most once.
  \label{as_mesh1}
\end{assumption}
\begin{assumption}
  For any element $K \in \MThG$, we can assign an interior element
  $K^{\circ} \in \Delta(K) \cap \MThc$, where $\Delta(K) := \{ K' \in
  \MTh \ | \ \overline{K'} \cap \overline{K} \neq \varnothing\}$
  denotes the set of elements that touch $K$.
  \label{as_mesh2}
\end{assumption}
By the quasi-uniformity of $\MTh$, there exists a
constant $C_{\Delta}$ independent of $h$ such that for any element
$K$, the set $\Delta(K) \subset B(\bm{x}_K, C_\Delta h_K)$, where
$B(\bm{z}, r)$ denotes the ball centered at the point $\bm{z}$ with
the radius $r$, and $\bm{x}_K$ is the barycenter of the element $K$.
In addition, we assume the open bounded domain $\Omega^*$ contains the
union of all balls $B(\bm{x}_K, C_\Delta h_K)(\forall K \in \MTh)$.

We introduce the jump and average operators which are widely used in
the discontinuous Galerkin framework. Let $f \in
\MFhI$ be an interior face shared by two neighbouring elements
$K^+$ and $K^-$, with the unit outward normal vectors $\un^+$ and
$\un^-$ on $f$, respectively. For any piecewise smooth scalar-valued
function $v$ and any piecewise smooth vector-valued function $\bm{q}$,
the following jump operators $\jump{\cdot}$ and average operators
$\aver{\cdot}$ are involved in our scheme: 
\begin{displaymath}
  \begin{aligned}
    \jump{v} &:= \un^+v^+|_f  + \un^- v^-|_f, \\
    \jump{\un \cdot \bm{q}} := \un^+ \cdot \bm{q}^+|_f &+ \un^- \cdot
    \bm{q}^-|_f, \quad \jump{\un \times \bm{q}} := \un^+ \times
    \bm{q}^+|_f + \un^- \times \bm{q}^-|_f
  \end{aligned}
\end{displaymath}
and 
\begin{displaymath}
  \aver{v} := \frac{1}{2}( v^+|_f +  v^-|_f), \quad \aver{\bm{q}} :=
  \frac{1}{2}(\bm{q}^+|_f + \bm{q}^-|_f),
\end{displaymath}
where $v^+ := v|_{K^+}$, $v^- := v|_{K^-}$, $\bm{q}^+ :=
\bm{q}|_{K^+}$, $\bm{q}^- := \bm{q}|_{K^-}$.  The jump operators
$\jump{\cdot}$ and the averages $\aver{\cdot}$ on $\Gamma$ will also
be used, and their definitions are modified as 
\begin{displaymath}
  \begin{aligned}
    \jump{v}|_{\Gamma_K} &:=  \un v|_{\Gamma_K}, \quad \jump{\un \cdot
    \bm{q}}|_{\Gamma_K} := \un \cdot \bm{q}|_{\Gamma_K}, \quad
    \jump{\un \times \bm{q}}|_{\Gamma_K} := \un \times
    \bm{q}|_{\Gamma_K}, \\ 
    & \aver{v}|_{\Gamma_K} := v|_{\Gamma_K}, \quad 
    \aver{\bm{q}}|_{\Gamma_K} := \bm{q}|_{\Gamma_K}, \\ 
  \end{aligned}
\end{displaymath}
for any element $K \in \MThG$, where $\un$ is the unit outward normal
on $\Gamma$.

In this paper, $C$ and $C$ with subscripts are denoted as generic
constants which may differ between lines but are always independent of
the mesh size $h$ and the location of the boundary $\Gamma$ relative
to the mesh. For a bounded domain $D$, we will follow the standard
notations to the Sobolev spaces $L^2(D)$, $H^s(D)$, $L^2(D)^d$,
$H^s(D)^d$ with the regular exponent $s \geq 0$,  and their
corresponding inner products, norms and seminorms. The  spaces
$H^s(\div, D)$ and $H^s(\curl, D)$ are also involved as well as their
associated norms and semi-norms.  We define $H_0^s(\curl, D)$ as the
space of functions in $H^s(\curl, D)$ with vanishing tangential
traces.  

Next, we introduce the local extension operator
\cite{Yang2021unfitted}, which will be
involved in the numerical scheme to ensure the stability of our
method.
Given an integer $m \geq 0$ and for element $K \in \MTh$, the
local operator $E_K$ is defined as
\begin{equation}
  \begin{aligned}
    E_K: L^2(K) & \rightarrow \mb{P}_m(B(\bm{x}_K, C_\Delta
    h_K)), \\
    {v} & \rightarrow E_K {v}, \\
  \end{aligned} 
  \quad (E_K {v})|_K = \pick {v}, \quad \forall {v} \in
  L^2(K),\\
  \label{eq_EK}
\end{equation}
where $\pick$ is the $L^2$ projection operator from $L^2(K)$ to
$\mb{P}_m(K)$. For $v \in L^2(K)$, $E_K$
extends its $L^2$ projection $\pick v$ to define on the ball
$B(\bm{x}_K, C_\Delta h_K)$. Particularly for $v \in \mb{P}_m(K)$,
$E_K v$ is just the direct extension of $v$ to the ball $B(\bm{x}_K,
C_\Delta h_K)$.  

Let us give some basic properties of the operator $E_K$.
\begin{lemma}
  There hold 
  \begin{equation}
    \begin{aligned}
      \| E_K v \|_{L^2(B(\bm{x}_K, C_\Delta h_K))} &\leq C \|
      \pick v \|_{L^2(K)},  \\
    \end{aligned} \quad \forall v \in L^2(K), \quad \forall K \in
    \MTh,
    \label{eq_EkL2}
  \end{equation}
  and 
  \begin{equation}
    \begin{aligned}
      \| E_K v \|_{L^2(B(\bm{x}_K, C_\Delta h_K))} &\leq C \| 
      v \|_{L^2(K)},  \\
      \|\nabla  E_K v \|_{L^2(B(\bm{x}_K, C_\Delta h_K))} &\leq C \|
      \nabla  v \|_{L^2(K)}. \\
    \end{aligned} \quad \forall v \in \mb{P}_m(K), \quad \forall K \in
    \MTh.
    \label{eq_EkL2p}
  \end{equation}
  \label{le_EkL2p}
\end{lemma}
\begin{proof}
  We mainly prove the $L^2$ estimate \eqref{eq_EkL2}. From the
  quasi-uniformity of the mesh, the ball $B(\bm{x}_K, \rho_K) \subset K$
  and there exists a constant $C_1$ such that $h_K \leq C_1 \rho_K$.
  By the norm equivalence restricted on the space $\mb{P}_m(\cdot)$,
  we have that
  \begin{displaymath}
    \|{q} \|_{L^2(B(\bm{x}_K, C_\Delta C_1 ))} \leq C \|{q}
    \|_{L^2(B(\bm{x}_K, 1))}, \quad \forall {q} \in
    \mb{P}_m(B(\bm{x}_K, C_\Delta C_1 )).
  \end{displaymath}
  Considering the affine mapping from the ball $B(\bm{x}_K, 1)$ to
  $B(\bm{x}_K, \rho_K)$, which simultaneously maps the ball
  $B(\bm{x}_K, C_\Delta C_1 )$ to $B(\bm{x}_K, C_\Delta C_1 \rho_K)$,
  we conclude that 
  \begin{displaymath}
    \begin{aligned}
      \| E_K v \|_{L^2(B(\bm{x}_K, C_\Delta h_K))} &\leq  \| E_K v
      \|_{L^2(B(\bm{x}_K, C_\Delta C_1 \rho_K))} \leq C \| \pick v
      \|_{L^2(B(\bm{x}_K,  \rho_K))} \leq C \|\pick v \|_{L^2(K)}. \\
    \end{aligned}
  \end{displaymath}
  The estimates in \eqref{eq_EkL2p} are then the direct consequences
  of \eqref{eq_EkL2}. This completes the proof.
\end{proof}
The stability of the curl operator near the boundary is also
guaranteed by the local operator.
To this end, we extend the operator $E_K$ to
vector-valued functions:
\begin{equation}
  \begin{aligned}
    E_K: L^2(K)^d & \rightarrow \mb{P}_m(B(\bm{x}_K, C_\Delta
    h_K))^d, \\
    \bm{v} & \rightarrow E_K \bm{v}, \\
  \end{aligned} 
  \quad (E_K \bm{v})|_K = \pick \bm{v}, \quad \forall \bm{v} \in
  L^2(K)^d,\\
  \label{eq_vEK}
\end{equation}
where $\pick$ is still the $L^2$ projection from $L^2(K)^d$ to the
polynomial space $\mb{P}_m(K)^d$. $E_K$ in \eqref{eq_vEK} can be
regarded as acting on vector-valued functions in a columnwise manner
with \eqref{eq_EK}.  Hence, Lemma \ref{le_EkL2p} also hold for
vector-valued functions. 
\begin{lemma}
  There hold 
  \begin{equation}
    \| E_K \bm{v} \|_{L^2(B(\bm{x}_K, C_\Delta h_K))} \leq C \| \Pi_K
    \bm{v} \|_{L^2(K)}, \quad \forall \bm{v} \in L^2(K)^d, \quad
    \forall K \in \MTh,
    \label{eq_vEKL2}
  \end{equation}
  and
  \begin{equation}
    \| \nabla \times E_K \bm{v} \|_{L^2(B(\bm{x}_K, C_\Delta h_K))}
    \leq C \| \nabla \times \bm{v} \|_{L^2(K)}, \quad \forall \bm{v}
    \in \mb{P}_m(K)^d \quad \forall K \in \MTh.
    \label{eq_vEKcurlp}
  \end{equation}
  \label{le_vEK}
\end{lemma}
In addition, for any piecewise smooth scalar(or vector)-valued
function $q_h$, we simply write $E_K(q_h|_K)$ as $E_K q_h$.

We define discontinuous/continuous piecewise polynomial spaces for
the partition $\MTh$, 
\begin{displaymath}
  \begin{aligned}
    \QhD &:= \{{v}_h \in L^2(\Oh) \ | \ {v}_h|_K \in \mb{P}_m(K), \
    \forall K \in \MTh \},  \quad
    \QhC := \QhD \cap H^1(\Oh), 
  \end{aligned}
\end{displaymath}
and we let the space $\Qh$ be either $\Qh := \QhD$ or $\Qh := \QhC$.
We then define the penalty operator $j_h(\cdot, \cdot)$ based on
$E_K$. The form $j_h(\cdot, \cdot)$ will be used to provide the
stabilization for the scalar-valued unknown. We note that this method
follows the idea of the ghost penalty method \cite{Burman2010ghost},
which extends the control of the relevant norms from the interior
domain to the entire computational domain \cite{Massing2019stabilized,
Burman2010ghost, Burman2012ficticious}.

We set $\qh := \Qh + H^1(\Omega)$, and define 
\begin{equation}
  j_h(p_h, q_h) := \sum_{K \in \MThG} \int_{K} (p_h - E_{K^{\circ}}
  p_h)(q_h - E_{K^{\circ}} q_h) \d{x}, \quad \forall p_h, q_h \in \qh,
  \label{eq_jh}
\end{equation}
and  introduce the corresponding seminorm $\jnorm{\cdot}$ with
\begin{displaymath}
  \jnorm{q_h}^2 := j_h(q_h, q_h), \quad \forall q_h \in \qh.
\end{displaymath}
Then we show the following properties of $j_h(\cdot, \cdot)$, which
are instrumental in the error estimation. 
\begin{lemma}
  There holds 
  \begin{equation}
    \| q_h \|_{L^2(\Ohc)}^2 + j_h(q_h, q_h) \leq C \| q_h
    \|_{L^2(\Oh)}^2 \leq C( \| q_h \|_{L^2(\Ohc)}^2 + j_h(q_h, q_h)),
    \quad \forall q_h \in \Qh.
    \label{eq_qhL2bound}
  \end{equation}
  \label{le_qhL2bound}
\end{lemma}
\begin{proof}
  Applying the triangle inequality immediately gives $j_h(q_h, q_h)
  \leq C \| q_h \|_{L^2(\Oh)}^2$. By Lemma \ref{le_EkL2p}, for any $K
  \in \MThG$, we obtain that
  \begin{displaymath}
    \begin{aligned}
      \|q_h \|_{L^2(K)}^2 & \leq C (\|q_h - E_{K^\circ} q_h
      \|_{L^2(K)}^2 +  \|E_{K^\circ} q_h \|_{L^2(K)}^2 ) \leq C (\|q_h
      - E_{K^\circ} q_h \|_{L^2(K)}^2 +  \|q_h \|_{L^2(K^{\circ})}^2
      ).
    \end{aligned}
  \end{displaymath}
  Summation over all elements in $\MThG$ immediately gives the desired
  estimate, which completes the proof.
\end{proof}
\begin{lemma}
  There holds 
  \begin{equation}
    \jnorm{v} \leq Ch^s \|v\|_{H^s(\Omega^*)}, \quad \forall v \in
    H^t(\Omega^*), 
    \label{eq_jnormv}
  \end{equation}
  where $s = \min(t, m), t \geq 1$.
  \label{le_jnormv}
\end{lemma}
\begin{proof}
  For any element $K \in \MThG$, there exists $p \in
  \mb{P}_m(B(\bm{x}_K, C_\Delta h_K))$ such that 
  \begin{displaymath}
    \| v - p \|_{L^2(K)} \leq C h^s \|v\|_{H^s(B(\bm{x}_K, C_\Delta
    h_K))}.
  \end{displaymath}
  Therefore, we have that
  \begin{displaymath}
    \begin{aligned}
      \|v - E_{K^\circ} v \|_{L^2(K)} \leq  \| v - p \|_{L^2(K)}
      +  \|p - E_{K^\circ} v \|_{L^2(K)} = \| v - p \|_{L^2(K)} + \|
      E_{K^{\circ}}( p -  v) \|_{L^2(K)},
    \end{aligned}
  \end{displaymath}
  and from Lemma \ref{le_EkL2p}, we find that
  \begin{displaymath}
    \begin{aligned}
     \| E_{K^{\circ}}( p -  v)
      \|_{L^2(K)} \leq C \| p - \Pi_{K^{\circ}} v \|_{L^2(K^{\circ})}
      \leq C (\| v - p \|_{L^2(K^{\circ})} + \|  v - \Pi_{K^{\circ}} v
      \|_{L^2(K^{\circ})}). 
    \end{aligned}
  \end{displaymath}
  Combing the approximation properties of $p$ and $\Pi_{K^{\circ}} v$,
  we arrive at the estimate \eqref{eq_jnormv}. This completes the
  proof.
\end{proof}
For the vector-valued unknown, we still define the 
discontinuous/continuous piecewise polynomial spaces on the partition
$\MTh$,
\begin{displaymath}
  \VhD :=  \{ \bm{v}_h \in L^2(\Oh)^d \ | \ \bm{v}_h|_K \in
  \mb{P}_m(K)^d, \ \forall K \in \MTh \}, \quad \VhC := \VhC \cap
  H^1(\Oh),
\end{displaymath}
and we let $\Vh$ be either $\Vh := \VhD$ or $\Vh := \VhC$. For the
space $\vh := \Vh + H^2(\Omega)^d$, we define the bilinear form
$s_h(\cdot, \cdot)$ as 
\begin{equation}
  g_h(\bm{u}_h, \bm{v}_h) := \sum_{K \in \MThG} \int_K (\nabla
  \times (\bm{u}_{h} -E_{K^{\circ}} \bm{u}_{h})) \cdot (\nabla
  \times (\bm{v}_{h} -E_{K^{\circ}} \bm{v}_{h})) \d{x}, \quad \forall
  \bm{u}_h, \bm{v}_h \in \vh,
  \label{eq_gh}
\end{equation}
and   define the corresponding seminorm $\gnorm{\cdot}$ as
\begin{displaymath}
  \gnorm{\bm{v}_h}^2 := g_h(\bm{v}_h, \bm{v}_h), \quad \forall
  \bm{v}_h \in \vh.
\end{displaymath}
We also give the following properties of $g_h(\cdot, \cdot)$.
\begin{lemma}
  There holds
  \begin{equation}
    \begin{aligned}
      C \big( \sum_{K \in \MTh} \|\nabla \times \bm{v}_h
      \|_{L^2(K)}^2 + &\gnorm{\bm{v}_h}^2 \big) \leq \sum_{K \in
      \MTh} \|\nabla \times\bm{v}_{h} \|_{L^2(K)}^2 \\ 
      \leq C & \big( \sum_{K \in \MTh} \|\nabla
      \times \bm{v}_h \|_{L^2(K)}^2 + \gnorm{\bm{v}_h}^2
      \big), \quad \forall \bm{v}_h \in \Vh.  \\
    \end{aligned}
    \label{eq_vhcurlbound}
  \end{equation}
  \label{le_vhcurlbound}
\end{lemma}
\begin{lemma}
  There holds 
  \begin{equation}
    \gnorm{ \bm{v} } \leq C h^s \| \bm{v} \|_{H^{s+1}(\Omega^*)},
    \quad \forall \bm{v} \in H^t(\Omega^*)^d,
    \label{eq_gnormv}
  \end{equation}
  where $s = \min(t - 1, r), t \geq 2$.
  \label{le_gnormv}
\end{lemma}
Lemma \ref{le_vhcurlbound} and Lemma \ref{le_gnormv} also follow from 
the triangle inequality and Lemma \ref{le_vEK}; see 
proofs of Lemma \ref{le_qhL2bound} and Lemma \ref{le_jnormv}.

We end this section by giving some basic results for unfitted methods.
The first is the trace estimate on the curve
\cite{Hansbo2002discontinuous, Huang2017unfitted}:
\begin{lemma}
  There exists a constant $h_0$, independent of $h$, such that for any
  $h \leq h_0$, there holds
  \begin{equation}
    \|w \|_{L^2(\Gamma_K)}^2 \leq C \left( h_K^{-1} \|w \|_{L^2(K)}^2
    + h_K \|w \|_{H^1(K)}^2 \right), \quad \forall w \in H^1(K), \quad
    \forall K \in \MThG.
    \label{eq_interfaceH1trace}
  \end{equation}
  \label{le_interfaceH1trace}
\end{lemma}
Hereafter, the condition $h \leq h_0$ is assumed to be always
fulfilled. In the error estimation, we require the Sobolev extension
theory \cite{Adams2003sobolev}: there exists an extension operator
$\Es: H^s(\Omega) \rightarrow H^s(\Omega^*)(s \geq
0)$ such that 
\begin{equation} 
  (\Es w)|_{\Omega} = w, \quad \|\Es w
  \|_{H^q(\Omega^*)} \leq C \|w \|_{H^q(\Omega)}, \quad 0 \leq
  q \leq s, \quad \forall w \in H^s(\Omega).
  \label{eq_Eextension}
\end{equation}

\section{Numerical Scheme to the Maxwell Problem in Smooth Domain}
\label{sec_scheme}
In this section, we are concerned with the time-harmonic Maxwell
problem defined on the smooth domain $\Omega$, which seeks the vector
field $\bm{u}$ and the scalar unknown (Lagrangian multiplier) $p$ such
that 
\begin{equation}
  \begin{aligned}
    \nabla \times (\coec \nabla \times \bm{u}) - \coek \coed \bm{u}
    - \coed \nabla p &= \bm{j}, && \text{in } \Omega, \\ 
    \nabla \cdot (\coed \bm{u}) &= 0, &&\text{in } \Omega, \\
    p  = 0, \quad \un \times \bm{u} &= \bm{g}, && \text{on } \Gamma, \\
  \end{aligned}
  \label{eq_Maxwell}
\end{equation}
where  $\bm{j}$ is the external source field and $\bm{g}$ is the
prescribed tangential trace. The magnetic permeability $\mu_r$ and the
electric permittivity $\coed$ are assumed to be $C^2$-smooth and
satisfy 
\begin{equation}
  0 < \mu_* \leq \mu(\bm{x}) \leq \mu^*, \quad 0 < \varepsilon_* \leq
  \varepsilon(\bm{x}) \leq \varepsilon^*, \quad \forall \bm{x} \in
  \overline{\Omega}.
  \label{eq_mue}
\end{equation}
We first show the well-posedness of the problem \eqref{eq_Maxwell}.
\begin{theorem}
  Assume that $k$ is not a Maxwell eigenvalue, then for
  $\bm{j} \in L^2(\Omega)^d$ and the boundary data
  $\bm{g} = \bm{0}$, the problem \eqref{eq_Maxwell} admits a unique
  solution $(\bm{u}, p) \in H^2(\Omega)^d \times H^1(\Omega)$ such
  that 
  \begin{equation}
    \|\bm{u} \|_{H^2(\Omega)} \leq \Cr \| \bm{j} \|_{L^2(\Omega)},
    \quad \|p \|_{H^1(\Omega)} \leq C  \| \bm{j}
    \|_{L^2(\Omega)},
    \label{eq_up}
  \end{equation}
  where $\Cr$ depends on $k$.
  \label{th_regularity}
\end{theorem}
\begin{proof}
  From the Helmholtz decomposition \cite{Girault1986finite}, $\bm{j}
  \in L^2(\Omega)^d$ can be decomposed as $\bm{j} = \bm{J} + \nabla
  q$, where $\bm{J}$ is a divergence-free field and $q \in
  H_0^1(\Omega)$ is the solution to the problem $ \Delta q = \nabla
  \cdot \bm{j}$ in $ \Omega$ such that $\|q \|_{H^1(\Omega)} \leq C
  \|\bm{j} \|_{L^2(\Omega)}$.  Thanks to the orthogonality of the
  Helmholtz decomposition,  $p$ is the solution of the elliptic
  problem  
  \begin{displaymath}
    -\nabla \cdot (\coed \nabla p) = -\Delta q \quad \text{in }
    \Omega, \quad p = 0 \quad \text{on } \Gamma. 
  \end{displaymath}
  By the condition \eqref{eq_mue}, we know that $p \in H_0^1(\Omega)$
  with $\|p\|_{H^1(\Omega)} \leq C \|q\|_{H^1(\Omega)} \leq C \|\bm{j}
  \|_{L^2(\Omega)}$.  Further, $\bm{u}$ is the solution of the
  system
  \begin{displaymath}
    \begin{aligned}
      \nabla \times (\coec \nabla \times \bm{u}) - \coek \coed \bm{u}
      & = \bm{J}, && \text{in } \Omega, \\
      \nabla \cdot (\coed \bm{u}) &= 0, &&\text{in } \Omega, \\
      \un \times \bm{u} &= \bm{0}, && \text{on } \Gamma. \\
    \end{aligned}
  \end{displaymath}
  From \cite[Proposition 1]{Perugia2002stabilized} and \cite[Section
  3]{Girault1986finite}, one has that $\bm{u} \in H_0^1(\curl,
    \Omega)$ with $\| \bm{u} \|_{H^1(\curl, \Omega)} \leq C(k)
    \|\bm{j} \|_{L^2(\Omega)}$. 
  In addition, the condition $\nabla \cdot (\coed \bm{u}) = 0$ brings
  that $\varepsilon \nabla \cdot \bm{u} = - \bm{u} \cdot \nabla
  \varepsilon$, together with \eqref{eq_mue}, which implies $\nabla
  \cdot \bm{u} \in H^1(\Omega)$.  From \cite[Corollary
  2.15]{Amrouche1998vector}, we conclude that $\bm{u} \in H^2(\Omega)$
  with the estimate $\| \bm{u} \|_{H^2(\Omega)} \leq C(k) \| \bm{j}
  \|_{L^2(\Omega)}$.
  The above results give the estimate \eqref{eq_up} and complete the
  proof.
\end{proof}
\begin{remark}
  The regularity result in Theorem \ref{th_regularity} may be invalid
  for the polygonal (polyhedral) domain because the embedding
  $H^1_0(\curl, \Omega) \cap H^1(\div, \Omega) \hookrightarrow
  H^2(\Omega)$ requires the domain $\Omega$ to be of class $C^2$
  \cite{Amrouche1998vector, Girault1986finite}. 
  \label{re_regularity}
\end{remark}
\begin{remark}
  Throughout this paper, the constants $C$ and $C_i$ are independent
  of $k$ unless otherwise stated. 
  \label{re_Ck}
\end{remark}

In this section, we propose a mixed numerical scheme for the Maxwell
problem \eqref{eq_Maxwell}, which reads: seek $(\bm{u}_h, p_h) \in \Vh
\times \Qh(r \geq m + 1, m \geq 0)$ such that 
\begin{equation}
  \begin{aligned}
    a_h(\bm{u}_h, \bm{v}_h) + b_h(\bm{v}_h, p_h) - \coek (\bm{u}_h,
    \bm{v}_h) + g_h(\bm{u}_h, \bm{v}_h)  &= l_h(\bm{v}_h), && \forall
    \bm{v}_h \in \Vh, \\
    b_h(\bm{u}_h, q_h) - c_h(p_h, q_h) - j_h(p_h, q_h) & = 0, &&
    \forall q_h \in \Qh. \\
  \end{aligned}
  \label{eq_mixedform}
\end{equation}
The bilinear form $a_h(\cdot, \cdot)$ is defined as 
\begin{align}
  &a_h(\bm{u}_h, \bm{v}_h) := a_{h, 0}(\bm{u}_h, \bm{v}_h)  + a_{h,
  1}(\bm{u}_h, \bm{v}_h) , \label{eq_bilinearA} 
  \end{align}
  where
  \begin{align*}
  a_{h, 0}(\bm{u}_h, \bm{v}_h) &:= \sum_{K \in \MTh} \int_{K^0}
  \coec (\nabla \times \bm{u}_h) \cdot (\nabla \times \bm{v}_h)
  \d{x} + \sum_{K \in \MTh} \int_{K^0} \nabla \cdot (\coed
  \bm{u}_h)\nabla \cdot (\coed \bm{v}_h) \d{x} \nonumber \\ 
  -  \sum_{f \in \MFhI} &\int_{f^0} \left( \aver{\coec
  \nabla \times \bm{u}_h} \cdot \jump{\un \times \bm{v}_h} +
  \aver{\coec \nabla \times \bm{v}_h} \cdot \jump{\un \times
  \bm{u}_h} \right) \d{s} \nonumber \\ 
  -  \sum_{K \in \MThG} &\int_{\Gamma_K} \left( \aver{\coec \nabla
  \times \bm{u}_h}\cdot \jump{\un \times \bm{v}_h}  + \aver{\coec
  \nabla \times \bm{v}_h}\cdot  \jump{\un \times \bm{u}_h} \right)
  \d{s},\nonumber  \\ 
  a_{h, 1}(\bm{u}_h, \bm{v}_h) &:=  \sum_{f \in \MFhI} \int_{f^0}
  \alpha h_f^{-1} \jump{\un \times \bm{u}_h} \cdot \jump{\un \times
  \bm{v}_h} \d{s} + \sum_{K \in \MThG} \int_{\Gamma_K} \alpha
  h_K^{-1} \jump{\un \times \bm{u}_h} \cdot \jump{\un \times
  \bm{v}_h} \d{s} \nonumber \\
  +& \sum_{f \in \MFhI} \int_{f^0} h_f^{-1} \jump{\un \cdot
  (\coed \bm{u}_h)} \cdot \jump{\un \cdot (\coed \bm{v}_h)}  \d{s},
  \nonumber 
\end{align*}
for $\forall \bm{u}_h, \bm{v}_h \in \vh$, where
$\alpha > 0$ is the penalty parameter that will be specified later.
The bilinear form $b_h(\cdot, \cdot)$ is defined as 
\begin{equation}
  \begin{aligned}
    b_h(\bm{v}_h, p_h) &:=  \sum_{K \in \MTh} \int_{K^0}
    \nabla \cdot(\coed \bm{v}_h) p_h \d{x} - \sum_{f \in \MFhI}
    \int_{f^0} \jump{\un \cdot (\coed \bm{v}_h)} \aver{p_h}
    \d{s}, \\
  \end{aligned}
  \label{eq_bilinearB}
\end{equation}
for $\forall \bm{v}_h \in \vh, \forall p_h \in \qh$.  The bilinear
form $c_h(\cdot, \cdot)$ is defined as 
\begin{equation}
  \begin{aligned}
    c_{h}(p_h, q_h) & := \sum_{f \in \MFhI} \int_f h_f
    \jump{p_h} \cdot \jump{q_h} \d{s} + \sum_{K \in \MThG}
    \int_{\Gamma_K} h_K \jump{p_h} \cdot \jump{q_h} \d{s}, \\
  \end{aligned}
  \label{eq_bilinearC}
\end{equation}
for $\forall p_h, q_h \in \qh$. The linear form
$l_h(\cdot)$ is defined as 
\begin{displaymath}
  \begin{aligned}
    l_h(\bm{v}_h) := \sum_{K \in \MTh} \int_{K^0} \bm{j} \cdot
    \bm{v}_h \d{x} - \sum_{K \in \MThG} \int_{\Gamma_K} (\aver{\coec
    \nabla \times \bm{v}_h} \cdot \bm{g} + \alpha h_f^{-1} \jump{\un
    \times \bm{v}_h} \cdot \bm{g}) \d{s},
  \end{aligned}
\end{displaymath}
for $\forall \bm{v}_h \in \vh$.  The forms $a_{h,0}(\cdot, \cdot)$ and
$a_{h, 1}(\cdot, \cdot)$ are defined to weakly impose the continuity
condition and the boundary condition of the exact solution. 

\begin{remark}
  Our error estimation is established in the case that $\Vh = \VhD$
  and $\Qh = \QhD$. We note that the numerical scheme also allows $\Vh
  = \VhC$ and $\Qh = \QhC$. For this case, the forms $a_h(\cdot,
  \cdot), b_h(\cdot, \cdot), c_h(\cdot, \cdot)$ can be further
  simplified since $\jump{\un \times \bm{v}_h}|_{\MFhI} = \bm{0}$,
  $\jump{\un \cdot (\coed \bm{v}_h)}|_{\MFhI} = 0$ for $\forall
  \bm{v}_h \in \Vh$, and  $\jump{q_h}|_{\MFhI} = \bm{0}$ for $\forall
  q_h \in \Qh$. 
\end{remark}

We first present the error estimates of the numerical solution to
\eqref{eq_mixedform}.
\begin{theorem}
  Let $(\bm{u}, p) \in H^{t + 1}(\Omega)^d \times H^{t}(\Omega)(t \geq
  1)$ be the exact solution to the problem \eqref{eq_Maxwell}, and let
  $a_h(\cdot, \cdot)$ be defined with a
  sufficiently large $\alpha$, and let $(\bm{u}_h, p_h) \in \Vh \times
  \Qh(r \geq m + 1, m \geq 0)$ be the numerical solution, then there
  exists a constant $h_1$ depending on $\Cr$ such that for any $h \leq
  h_1$, there hold
  \begin{equation}
    \begin{aligned}
      \Anorm{\bm{u} - \bm{u}_h} + \Cnorm{p - p_h} &\leq C_0 h^{s} ( \|
      \bm{u} \|_{H^{s+1}(\Omega)} + \| p \|_{H^s(\Omega)}),
      \\
      \| \bm{u} - \bm{u}_h \|_{L^2(\Omega)} & \leq C_1 h^{s+1} ( \|
      \bm{u} \|_{H^{s+1}(\Omega)} + \| p \|_{H^s(\Omega)}),
    \end{aligned}
    \label{eq_mixederror}
  \end{equation}
  \label{th_errorestimate}
  where $s = \min(m + 1, t)$ and $C_1$ depends on $\Cr$. The
  definitions of energy norms $\Anorm{\cdot}$ and $\Cnorm{\cdot}$ are
  given in the next section.
\end{theorem}
The uniqueness of the problem 
\eqref{eq_mixedform} directly follows from the results in Theorem
\ref{th_errorestimate}.
\begin{corollary}
  Let $a_h(\cdot, \cdot)$ be defined  with a sufficiently large
  $\alpha$, the mixed formulation \eqref{eq_mixedform} admits a unique
  solution provided $h \leq h_1$.
\end{corollary}
For the error estimation, we write \eqref{eq_mixedform} into the
following equivalent formulation: seek $(\bm{u}_h, p_h) \in \Vh \times
\Qh$ such that
\begin{equation}
  \begin{aligned}
    A_h(\bm{u}_h, p_h; \bm{v}_h, q_h) + B_h(\bm{v}_h, q_h; p_h) -
    \coek (\bm{u}_h, \bm{v}_h) + g_h(\bm{u}_h, \bm{v}_h) &=
    l_h(\bm{v}_h), && \forall (\bm{v}_h, q_h) \in \Vh \times \Qh, \\
    B_h(\bm{u}_h, p_h; q_h) & = 0, && \forall q_h \in \Qh,  \\
  \end{aligned}
  \label{eq_amixedform}
\end{equation}
where
\begin{displaymath}
  \begin{aligned}
    A_h(\bm{u}_h, p_h; \bm{v}_h, q_h) &:= a_h(\bm{u}_h, \bm{v}_h) +
    c_h(p_h, q_h) + j_h(p_h, q_h), \\
    B_h(\bm{v}_h, q_h; p_h) &:= b_h(\bm{v}_h, p_h) - c_h(p_h, q_h) -
    j_h(p_h, q_h).
  \end{aligned}
\end{displaymath}

\section{Error Estimate for the Mixed Formulation}
\label{sec_erroraux}
In this section, we derive the error estimates by analysing the
problem \eqref{eq_amixedform}. We begin by introducing energy norms: 
for the space $\vh$, we define the seminorm
\begin{displaymath}
  \begin{aligned}
    \aseminorm{\bm{v}_h}^2 &:= \sum_{K \in \MTh} \|\nabla \times
    \bm{v}_h \|_{L^2(K^0)}^2 + \sum_{K \in \MTh} \| \nabla \cdot
    (\coed \bm{v}_h) \|_{L^2(K^0)}^2 + \sum_{f \in \MFhI} h_f^{-1} \|
    \jump{\un \times \bm{v}_h} \|_{L^2(f^0)}^2  \\
    +& \sum_{f \in \MFhI}  h_f^{-1} \| \jump{\un \cdot (\coed
    \bm{v}_h)} \|_{L^2(f^0)}^2  + \sum_{K \in \MThG} h_K^{-1} \|
    \jump{\un \times \bm{v}_h} \|_{L^2(\Gamma_K)}^2,  \quad \forall
    \bm{v}_h \in \vh, \\
  \end{aligned}
\end{displaymath}
and norms: 
\begin{displaymath}
  \anorm{\bm{v}_h}^2 := \coek \| \bm{v}_h \|_{L^2(\Omega)}^2  +
  \aseminorm{\bm{v}_h}^2+ \gnorm{\bm{v}_h}^2,
\end{displaymath}
\begin{displaymath}
  \begin{aligned}
    \Anorm{\bm{v}_h}^2 := \anorm{\bm{v}_h}^2 &+ \sum_{f \in \MFhI} h_f
    \| \aver{\nabla \times \bm{v}_h} \|_{L^2(f^0)}^2 + \sum_{K \in
    \MThG} h_K \| \aver{\nabla \times \bm{v}_h} \|_{L^2(\Gamma_K)}^2,
  \end{aligned}
\end{displaymath}
for $\forall \bm{v}_h \in \vh$. For the space $\qh$, we define
the seminorm 
\begin{displaymath}
  \cseminorm{q_h}^2 := \sum_{f \in \MFhI} h_f \| \jump{q_h}
  \|_{L^2(f^0)}^2   + \sum_{K \in \MThG} h_K \| \jump{q_h}
  \|_{L^2(\Gamma_K)}^2+  \jnorm{q_h}^2, \quad \forall q_h \in \qh,
\end{displaymath}
and norms 
\begin{displaymath}
  \begin{aligned}
    \cnorm{q_h}^2 := \| q_h \|_{L^2(\Omega)}^2 +   \cseminorm{q_h}^2,
    \quad
    \Cnorm{q_h}^2  :=  \cnorm{q_h}^2 +  \sum_{f \in \MFhI} h_f  \|
    \aver{q_h} \|_{L^2(f^0)}^2, 
  \end{aligned}
\end{displaymath}
for $\forall q_h \in \qh$.  For $\forall (\bm{v}_h, q_h) \in \Vh
\times \Qh$, we define 
\begin{displaymath}
  \begin{aligned}
    \Wnorm{\bm{v}_h, q_h}^2 := \Anorm{\bm{v}_h}^2 + \Cnorm{q_h}^2.
  \end{aligned}
\end{displaymath}

From Lemma \ref{le_vhcurlbound} and Lemma \ref{le_qhL2bound}, we state
the following equivalence results of above energy norms. 
\begin{lemma}
  There holds
  \begin{equation}
    \anorm{\bm{v}_h} \leq \Anorm{\bm{v}_h} \leq C \anorm{\bm{v}_h},
    \quad \forall \bm{v}_h \in \Vh.
    \label{eq_aAnorm}
  \end{equation}
  \label{le_aAnorm}
\end{lemma}
\begin{proof}
  From the standard trace estimate, the estimate 
  \eqref{eq_interfaceH1trace}, and  Lemma
  \ref{le_vhcurlbound}, we derive that 
  \begin{displaymath}
    \begin{aligned}
      \sum_{f \in \MFhI} h_f \| \aver{\nabla \times \bm{v}_h}
      \|_{L^2(f^0)}^2 &\leq \sum_{f \in \MFhI}  h_f \|
      \aver{\nabla \times \bm{v}_{h} }\|_{L^2(f)}^2  
      \leq C \sum_{K \in \MTh} \| \nabla \times \bm{v}_{h}
      \|_{L^2(K)}^2 \leq C \anorm{\bm{v}_h}^2, \\
    \end{aligned}
  \end{displaymath}
  and 
  \begin{displaymath}
    \begin{aligned}
      \sum_{K \in \MThG} h_K \| \aver{\nabla \times \bm{v}_h}
      \|_{L^2(\Gamma_K)}^2 &\leq C \sum_{K \in \MThG} \| \nabla \times
      \bm{v}_{h} \|_{L^2(K)}^2 \leq C \anorm{\bm{v}_h}^2. \\
    \end{aligned}
  \end{displaymath}
  The above estimates bring us the estimate
  \eqref{eq_aAnorm}. This completes the proof.
\end{proof}
\begin{lemma}
  There holds
  \begin{equation}
    \cnorm{q_h} \leq \Cnorm{q_h} \leq C  \cnorm{q_h} , \quad \forall
    q_h \in \Qh.
    \label{eq_cCnorm}
  \end{equation}
  \label{le_cCnorm}
\end{lemma}
\begin{proof}
  The proof is similar to the proof of Lemma \ref{le_aAnorm}.
\end{proof}
We show the equivalence between $\anorm{\cdot}$ and $\Anorm{\cdot}$
restricted on the space $\Vh$, and also $\Cnorm{\cdot}$ and
$\cnorm{\cdot}$ are equivalent on the space $\Qh$.
The two norms $\anorm{\cdot}$ and $\cnorm{\cdot}$ are more natural for
the analysis when dealing with piecewise polynomial spaces $\Vh$ and
$\Qh$.

The derivation to the error estimation for the mixed formulation
\eqref{eq_amixedform} is decoupled into several steps.
\subsection*{Inf-Sup Stability}
We first prove an inf-sup stability for the bilinear form $B_h(\cdot;
\cdot)$, which is crucial for the mixed formulation. The proof
requires some stability results of spaces defined on the interior
domain $\Ohc$. 
For the partition $\MThc$, we introduce the following
piecewise polynomial spaces:
\begin{displaymath}
  \begin{aligned}
    \Qhc := \{ q_h \in L^2(\Ohc) \ | \ q_h \in \mb{P}_m(K), \ 
    \forall K \in \MThc \}, \quad \Qhcc := \Qhc \cap H^1(\Ohc), 
  \end{aligned}
\end{displaymath}
\begin{displaymath}
  \Vhc := \{ \bm{v}_h \in L^2(\Ohc)^d \ | \ \bm{v}_h \in
  \mb{P}_m(K)^d, \ \forall K \in \MThc \}, \quad \Vhcc := \Vhc \cap
  H^1(\Ohc)^d. 
\end{displaymath}
The inf-sup stability of our method is based on the following result:
\begin{lemma}
  For $m \geq 1$, there holds
  \begin{equation}
    \sup_{\bm{v}_h \in \Vhcc \cap
    H_0^1(\Ohc)^d} \frac{\int_{\Ohc} \nabla \cdot (\coed \bm{v}_h)
    q_h \d{x} }{\|\bm{v}_h \|_{H^1(\Ohc)}} \geq C \|q_h
    \|_{L^2(\Ohc)} , \quad \forall q_h \in \Qhcc.
    \label{eq_infsupTH}
  \end{equation}
  \label{le_infsupTH}
\end{lemma}
The inf-sup condition for the divergence constraint on the interior
domain $\Ohc$ has been established in \cite{Guzman2018infsup}, and
this estimate \eqref{eq_infsupTH} is a modification of \cite[Theorem
1]{Guzman2018infsup}. The proof is included in Appendix
\ref{sec_app_proof}.

Now, we are ready to prove the inf-sup stability. 
\begin{theorem}
  There holds
  \begin{equation}
    \sup_{(\bm{v}_h, r_h) \in \Vh \times \Qh}
    \frac{B_h(\bm{v}_h, r_h; q_h)}{ \Wnorm{\bm{v}_h,
    r_h }} \geq C \Cnorm{q_h}, \quad \forall q_h \in \Qh.
    \label{eq_infsup}
  \end{equation}
  \label{th_infsup}
\end{theorem}
\def\wqc{\wt{q}_{h, \bmr{c}}}
\def\wqt{\wt{q}_{h, \bot}}
\def\qc{{q}_{h, \bmr{c}}}
\def\qt{{q}_{h, \bot}}
\begin{proof}
  We first prove the case $m \geq 1$. For fixed $q_h \in \Qh$, from
  \cite[Theorem 2.1]{Karakashian2007convergence}, there exists $\qc
  \in \QhC$ such that 
  \begin{equation}
    \| q_h - \qc \|_{L^2(\Oh)}^2 \leq C \sum_{f \in \MFhI} h_f \|
    \jump{q_h} \|_{L^2(f)}^2,
    \label{eq_L2qhqc}
  \end{equation}
  and we let $\qt := q_h - \qc$ and $\wqc := \qc|_{\Ohc} \in \Qhcc$. From Lemma
  \ref{le_infsupTH}, there exists $\wt{\bm{v}}_h \in \Vhcc \cap
  H_0^1(\Ohc)$ such that 
  \begin{displaymath}
    \|\wqc \|_{L^2(\Ohc)} \|\wt{\bm{v}}_h \|_{H^1(\Ohc)}\leq C \int_{\Ohc}
    \nabla \cdot (\coed \wt{\bm{v}}_h) \wqc \d{x},
  \end{displaymath}
  and $\wt{\bm{v}}_h$ is selected to  satisfy  $ \|\wt{\bm{v}}_h
  \|_{H^1(\Ohc)} = \|\wqc \|_{L^2(\Ohc)}$. Since $\wt{\bm{v}}_h \in
  H_0^1(\Ohc)$,  we extend it to the domain $\Oh$ by zero
  and there holds
  \begin{displaymath}
    \begin{aligned}
      B_h(\wt{\bm{v}}_h, \bm{0}; \qc) = \int_{\Ohc} \nabla
      \cdot (\coed \wt{\bm{v}}_h) \wqc \d{x} \geq C \|\wqc
      \|_{L^2(\Ohc)}^2 = C \| \qc \|_{L^2(\Ohc)}^2.
    \end{aligned}
  \end{displaymath}
  Further, for any $\delta > 0$,  we derive that
  \begin{displaymath}
    \begin{aligned}
      B_h(\wt{\bm{v}}_h, \bm{0}; \qt) &= \sum_{K \in \MThc}
      \int_{K} \nabla \cdot (\coed \wt{\bm{v}}_h) \qt \d{x} \geq -
      C\delta  \|\wt{\bm{v}}_h\|_{H^1(\Ohc)}^2 -
      \delta^{-1} \| \qt \|_{L^2(\Ohc)}^2 \\
      & \geq -C \delta \| \qc \|_{L^2(\Ohc)}^2 - \delta^{-1}
      \| \qt \|_{L^2(\Ohc)}^2.
    \end{aligned}
  \end{displaymath} 
  From \eqref{eq_L2qhqc}, we have that
  \begin{displaymath}
    B_h(\bm{0}, -q_h; q_h) \geq \sum_{f \in \MFhI}  h_f \| \jump{q_h}
    \|_{L^2(f)}^2 + \jnorm{q_h}^2 \geq C  \| \qt \|_{L^2(\Oh)}^2 +
    \jnorm{q_h}^2.
  \end{displaymath}
  We take $(\bm{v}_h, r_h) :=
  (\wt{\bm{v}}_h, -t q_h)$ with a parameter $t > 0$, we conclude
  that there exist constants $C_0, \ldots, C_3$ such that
  \begin{displaymath}
    \begin{aligned}
      B_h(&\bm{v}_h, r_h; q_h) = B_h(\bm{v}_h, \bm{0}; \qc)
      + B_h(\bm{v}_h, \bm{0}; \qt) + B_h(\bm{0}, -t q_h; q_h) \\
      & \geq C_0 \| \qc \|_{L^2(\Ohc)}^2 - C_2 \delta \| \qc
      \|_{L^2(\Ohc)}^2 - C_3\delta^{-1} \| \qt \|_{L^2(\Ohc)}^2 + tC_1
      \| \qt \|_{L^2(\Oh)}^2 + t\jnorm{q_h}^2 \\ 
      & \geq (C_0 - C_2\delta) \| \qc \|_{L^2(\Ohc)}^2 + (tC_1 -
      C_3\delta^{-1}) \| \qt \|_{L^2(\Ohc)}^2 + t \jnorm{q_h}^2. \\
    \end{aligned}
  \end{displaymath}
  Selecting suitable parameters $\delta, t$ and by Lemma
  \ref{le_qhL2bound}, we arrive at $B_h(\bm{v}_h, r_h; q_h) \geq
  C \cnorm{q_h}^2 \geq C \Cnorm{q_h}^2$. Moreover, we have that
  $\Wnorm{\bm{v}_h, r_h}^2 = \Anorm{\bm{v}_h}^2 + \Cnorm{r_h}^2 \leq C
  \| \qc \|_{L^2(\Ohc)}^2 +  C \Cnorm{q_h}^2 \leq C \Cnorm{q_h}^2$.
  Collecting all above estimates gives the desired inf-sup estimate
  \eqref{eq_infsup} for the case $m \geq 1$.

  Then we prove   the case $m = 0$. Since $\coed$ is smooth, there
  exists $\bm{v} \in H_0^1(\Omega)^d$ such that
  \cite{Monk2003finite}
  \begin{equation}
    \nabla \cdot (\coed \bm{v}) = q_h, \ \text{in } \Omega, \quad
    \bm{v} = \bm{0}, \ \text{on } \Gamma,
    \label{eq_coedv}
  \end{equation}
  with $\|\bm{v} \|_{H^1(\Omega)} \leq C \|{q}_h \|_{L^2(\Omega)}$.
  We extend $\bm{v}$ by zero to the domain $\Oho$ and let $\bm{v}_h \in
  \bmr{V}_h^1$ be its Scott-Zhang interpolant on the mesh $\MTh$. 
  We have that
  \begin{displaymath}
    \begin{aligned}
      B_h(\bm{v}_h, 0; q_h) = \int_{\Omega} \nabla \cdot (\coed
      \bm{v}_h) q_h \d{s} = \|q_h\|_{L^2(\Omega)}^2+ 
      \int_{\Omega} \nabla \cdot(\coed ( \bm{v}_h - \bm{v}) ) q_h
      \d{x}.
    \end{aligned}
  \end{displaymath}
  From the integration by parts and the approximation property of the
  interpolant, there holds
  \begin{displaymath}
    \begin{aligned}
      \int_{\Omega} \nabla \cdot(\coed ( \bm{v}_h - \bm{v}) ) q_h \d{x}
      = \sum_{f \in \MFhI} \int_{f^0} &(\coed (\bm{v}_h - \bm{v}))
      \cdot \jump{q_h} \d{s} + \sum_{K \in \MThG} \int_{\Gamma_K}
      (\coed (\bm{v}_h - \bm{v}))\cdot  \jump{q_h} \d{s}\\
      \geq -C_0 \delta \|\bm{v} \|_{H^1(\Omega)}^2- C_1
      &\delta^{-1}\Big( \sum_{f \in \MFhI} h_f \| \jump{q_h}
      \|_{L^2(f^0)}^2+ \sum_{K \in \MThG} h_K \| \jump{q_h}
      \|_{L^2(\Gamma_K)}^2 \Big),   \\ 
    \end{aligned}
  \end{displaymath}
  for any $\delta > 0$. Similarly, we let $r_h := -t q_h$ and take
  suitable $\delta$ and $t$, we have that
  \begin{displaymath}
    B_h(\bm{v}_h, r_h; q_h) \geq C \cnorm{q_h}^2 \geq C \Cnorm{q_h}^2.
  \end{displaymath}
  Clearly, one has that $\Wnorm{\bm{v}_h, r_h} \leq C \Cnorm{q_h}$.
  Combining all results leads to the inf-sup estimate
  \eqref{eq_infsup}, which completes the proof.
\end{proof}

\subsection*{Continuity and Coercivity Properties}
Next, we show the continuity and the coercivity
properties of the bilinear form $A_h(\cdot; \cdot)$ and $B_h(\cdot;
\cdot)$. 

\begin{lemma}
  There hold
  \begin{equation}
    |A_h(\bm{u}_h, p_h; \bm{v}_h, q_h)| \leq C
    \Wnorm{\bm{u}_h, p_h} \Wnorm{\bm{v}_h,q_h},
    \quad \forall (\bm{u}_h, p_h), (\bm{v}_h, q_h)
    \in \vh \times \qh, 
    \label{eq_Acontinuity}
  \end{equation}
  and 
  \begin{equation}
    |B_h(\bm{v}_h, r_h; q_h)| \leq C \Wnorm{\bm{v}_h,
    r_h} \Cnorm{q_h}, \quad \forall (\bm{v}_h, r_h)
    \in \vh \times \qh, \ \forall q_h \in \qh.
    \label{eq_Bcontinuity}
  \end{equation}
  \label{le_continuity}
\end{lemma}
\begin{proof}
  The proofs of the continuity results \eqref{eq_Acontinuity} and
  \eqref{eq_Bcontinuity} are quite formal, directly following from the
  Cauchy-Schwarz inequality and the definitions \eqref{eq_jh},
  \eqref{eq_gh} and \eqref{eq_bilinearA} - \eqref{eq_bilinearC}.
\end{proof}
Let us show the coercivity of the bilinear form $a_h(\cdot, \cdot)$.
\begin{lemma}
  Let the bilinear form $a_h(\cdot, \cdot)$ be defined with a
  sufficiently large $\alpha > 0$, then there holds
  \begin{equation}
    a_h(\bm{v}_h, \bm{v}_h) + g_h(\bm{v}_h, \bm{v}_h) \geq C
    \aseminorm{\bm{v}_h}^2,  \quad \forall \bm{v}_h \in \Vh .
    \label{eq_acoercivity}
  \end{equation}
  \label{le_acoercivity}
\end{lemma}
\begin{proof}
  From Lemma \ref{le_vhcurlbound} and the trace estimate, there holds 
  \begin{align*}
    -2\sum_{f \in \MFhI} &\int_{f^0} \aver{\coec \nabla
    \times \bm{v}_h} \cdot \jump{\un \times \bm{v}_h} \d{s} \\
     \geq& -C_0 t \sum_{f \in \MFhI} h_f \| \aver{\coec \nabla \times
    \bm{v}_h} \|_{L^2(f)}^2 - C_1 t^{-1} \sum_{f \in \MFhI} h_f^{-1}
    \| \jump{\un \times \bm{v}_h} \|_{L^2(f^0)}^2\\
    \geq & -C_0 t (\sum_{K \in \MTh} \| \nabla \times \bm{v}_h
    \|_{L^2(K^0)}^2  + \gnorm{\bm{v}_h}^2)  - C_1 t^{-1}
    \sum_{f \in \MFhI} h_f^{-1} \| \jump{\un \times \bm{v}_h}
    \|_{L^2(f^0)}^2, 
  \end{align*}
  for $\forall t > 0$. Similarly, from the trace estimate
  \eqref{eq_interfaceH1trace}, we have that
  \begin{align*}
    -2\sum_{K \in \MThG} &\int_{\Gamma_K} \aver{\coec \nabla \times
    \bm{v}_h} \cdot \jump{\un \times \bm{v}_h} \d{s}  \\
    \geq & -C_0 t (\sum_{K \in \MTh} \| \nabla \times \bm{v}_h
    \|_{L^2(K^0)}^2  + \gnorm{\bm{v}_h}^2)  - C_1 t^{-1}
    \sum_{K \in \MThG} h_K^{-1} \| \jump{\un \times \bm{v}_h}
    \|_{L^2(\Gamma_K)}^2, 
  \end{align*}
  for $\forall t > 0$. Thus, we conclude that there exist constants
  $C_0, \ldots, C_3$ such that 
  \begin{align*}
    a_h(\bm{v}_h, \bm{v}_h) + & g_h(\bm{v}_h, \bm{v}_h) \geq (C_0 -
    tC_1) \sum_{K \in \MTh} \| \nabla \times \bm{v}_h
    \|_{L^2(K^0)}^2 + \sum_{K \in \MTh}  \| \nabla \cdot (\coed
    \bm{v}_h) \|_{L^2(K^0)}^2  \\
    +  & (\alpha - C_2 t^{-1} ) \sum_{f \in \MFhI} h_f^{-1} \|
    \jump{\un \times \bm{v}_h} \|_{L^2(f^0)}^2 + \sum_{f
    \in \MFhI}  \| \jump{\un \cdot (\coed \bm{v}_h)}
    \|_{L^2(f^0)}^2 \\
    + & (\alpha - C_3 t^{-1} )\sum_{K \in \MThG} h_K^{-1} \|
    \jump{\un \times \bm{v}_h} \|_{L^2(\Gamma_K)}^2, \quad \forall t
    > 0.
  \end{align*}
  Taking a proper $t$ such that $C_0 - t C_1 > 0$ and choosing a
  sufficiently large $\alpha$, we obtain the estimate
  \eqref{eq_acoercivity}. This completes the proof.
\end{proof}
As a direct consequence of Lemma \ref{le_acoercivity} and Lemma
\ref{le_aAnorm},  we have that 
\begin{equation}
  A_h(\bm{v}_h, q_h; \bm{v}_h, q_h) + g_h(\bm{v}_h,
  \bm{v}_h) \geq C (
  \Anorm{\bm{v}_h}^2 + \cseminorm{q_h}^2) - C \coek \|\bm{v}_h
  \|_{L^2(\Omega)}^2, \quad \forall (\bm{v}_h, q_h) \in \Vh
  \times \Qh.
  \label{eq_Acoercivity}
\end{equation}

\subsection*{Error estimates} 
We are ready to derive the error bounds under both the energy norm and
the $L^2$ norm for the problem \eqref{eq_amixedform}. We first present
the approximation results under energy norms. Here we assume the mesh
size $h$ satisfies $hk \leq C$.
\begin{lemma}
  For any $\bm{v} \in H^{t+1}(\Omega)^d(t \geq 1)$, there exists
  $\bm{v}_h \in \Vh$ such that
  \begin{equation}
    \Anorm{\bm{v} - \bm{v}_h} \leq C h^{s} \|\bm{v}
    \|_{H^{s+1}(\Omega)}, \quad \gnorm{\bm{v}_h} \leq Ch^{s}
    \|\bm{v} \|_{H^{s+1}(\Omega)},
    \label{eq_Anormapp}
  \end{equation}
  where $s = \min(t, r)$.
  \label{le_Anormapp}
\end{lemma}
\begin{proof}
  Let $\bm{v}_h$ be the Lagrange interpolant of $\Es \bm{v}$ into the
  space $\Vh$, where $\Es$ is defined as \eqref{eq_Eextension}. For
  any element $K \in \MThG$, there exists $\wt{\bm{v}}_h \in
  \mb{P}_{r}(B(\bm{x}_{K^\circ}, C_\Delta h_{K^\circ}))^d$ such that 
  \begin{displaymath}
    \|\nabla \times (\Es \bm{v} - \wt{\bm{v}}_h)
    \|_{L^2(B(\bm{x}_{K^\circ}, C_\Delta h_{K^\circ}))} \leq C
    h_{K^\circ}^{s} \|\Es \bm{v} \|_{H^{s+1}(B(\bm{x}_{K^\circ},
    C_\Delta h_{K^\circ}))}.
  \end{displaymath}
  The term $\gnorm{\bm{v}_h}$ can be bounded as in Lemma
  \ref{le_gnormv}, i.e.
  \begin{align*}
    &\gnorm{\bm{v}_h}^2 =  \sum_{K \in \MThG} \|\nabla \times
    \bm{v}_h -   \nabla \times E_{K^\circ}\bm{v}_h \|_{L^2(K)}^2 \\
    & \leq C \sum_{K \in \MThG} ( \| \nabla \times \bm{v}_h - \nabla
    \times \wt{\bm{v}}_h \|_{L^2(K)}^2 +  \|\nabla \times
    \wt{\bm{v}}_h - \nabla \times E_{K_0^\circ}\bm{v}_h
    \|_{L^2(K)}^2)\\
    & \leq C h^{r} \|E^* \bm{v} \|_{H^{r+1}(\Omega^*)} +  C \sum_{K
    \in \MThG}  \|\nabla \times \wt{\bm{v}}_h - \nabla \times
    \bm{v}_h \|_{L^2(K^{\circ})}^2 \\
    & \leq C h^{r} \|E^* \bm{v}
    \|_{H^{r+1}(\Omega^*)}  \leq Ch^{r} \| \bm{v}
    \|_{H^{r+1}(\Omega)}. 
  \end{align*}
  The approximation errors $\anorm{\bm{v} - \bm{v}_h}$ and $\|\bm{v} -
  \bm{v}_h \|_{L^2(\Omega)}$ can be bounded in a standard procedure,
  by following the Sobolev extension operator \eqref{eq_Eextension},
  the approximation property of $\bm{v}_h$ and the inequality $kh \leq
  C$. Consequently, the approximation estimate \eqref{eq_Anormapp} is
  reached, which completes the proof.
\end{proof}
\begin{lemma}
  For any $q \in H^{t+1}(\Omega)(t \geq 0)$, there exists $q_h \in
  \Qh$ such that
  \begin{equation}
    \Cnorm{q - q_h} \leq C h^{s + 1} \|q\|_{H^{s+1}(\Omega)}, \quad
    \jnorm{q_h} \leq C h^{s+1} \| q \|_{H^{s+1}(\Omega)},
    \label{eq_cnormapp}
  \end{equation}
  where $s = \min(t, m)$.
  \label{le_cnormapp}
\end{lemma}
\begin{proof}
  The proof is similar to that of Lemma \ref{le_Anormapp}.
\end{proof}
We then state the Galerkin orthogonality of $A_h(\cdot; \cdot)$
and $B_h(\cdot; \cdot)$.
\begin{lemma}
  Let $(\bm{u}, p) \in H^2(\Omega)^d \times H^1(\Omega) $ be the exact
  solution to the problem \eqref{eq_Maxwell}, and let $(\bm{u}_h, p_h)
  \in \Vh \times \Qh$ be the numerical solution, then there hold
  \begin{equation}
    \begin{aligned}
      A_h(\bm{u} - \bm{u}_h, p-p_h; \bm{v}_h, q_h) + B_h(\bm{v}_h,
      q_h; p - p_h) - \coek(\bm{u} - \bm{u}_h, \bm{v}_h) & =
      g_h(\bm{u}_h, \bm{v}_h), &&\forall (\bm{v}_h, q_h) \in \Vh
      \times \Qh \\
      B_h(\bm{u} - \bm{u}_h, p-p_h; q_h) & = -j_h(p, q_h), &&\forall
       q_h \in \Qh.\\
    \end{aligned}
    \label{eq_Galerkinorth}
  \end{equation}
  \label{le_Galerkinorth}
\end{lemma}
\begin{proof}
  The equation \eqref{eq_Galerkinorth} follows from the continuity
  condition of the exact solution and the definitions of $A_h(\cdot;
  \cdot)$ and $B_h(\cdot; \cdot)$.
\end{proof}
Based on Theorem \ref{th_infsup}, Lemma \ref{le_continuity} and Lemma
\ref{le_acoercivity}, we can prove the main results of Theorem
\ref{th_errorestimate}, which are given in Theorems \ref{th_DGerror}
and \ref{th_L2error}. 
\begin{theorem}
  Let $(\bm{u}, p) \in H^{t + 1}(\Omega)^d \times H^{t}(\Omega)(t \geq
  1)$ be the exact solution to the problem \eqref{eq_Maxwell}, and let
  $a_h(\cdot, \cdot)$ be defined with a sufficiently large $\alpha$,
  and let $(\bm{u}_h, p_h) \in \Vh \times \Qh(r \geq m + 1, m \geq
  0)$ be the numerical solution, then there holds
  \begin{equation}
    \begin{aligned}
      \Anorm{\bm{u} - \bm{u}_h} + \Cnorm{p - p_h} \leq
      C h^{s} (\|\bm{u} \|_{H^{s+1}(\Omega)} + \|p_h
      \|_{H^{s}(\Omega)}) + Ck\|\bm{u} - \bm{u}_h \|_{L^2(\Omega)},
    \end{aligned} 
    \label{eq_DGerror}
  \end{equation}
  where $s = \min(t, m + 1)$. 
  \label{th_DGerror}
\end{theorem}
\begin{proof}
  We first take $(\bm{v}_h, q_h) \in \Ker(B_h) := \{
  (\bm{w}_h, \nu_h) \in \Vh \times \Qh \ | \ B_h(\bm{w}_h,
  {\nu}_h; r_h) = 0, \ \forall r_h \in \Qh \}$. 
  From the coercivity \eqref{eq_Acoercivity} and the
  orthogonality \eqref{eq_Galerkinorth}, we have that 
  \begin{displaymath}
    \begin{aligned}
      C (\Anorm{\bm{v}_h - \bm{u}_h&}^2 + \cseminorm{q_h - p_h}^2 -
      \coek \| \bm{v}_h - \bm{u}_h \|_{L^2(\Omega)}^2) \\
      \leq &A_h(\bm{v}_h - \bm{u}_h, {q}_h- p_h;
      \bm{v}_h - \bm{u}_h, q_h - p_h) + g_h(
      \bm{v}_h - \bm{u}_h, \bm{v}_h - \bm{u}_h) \\
       \leq & A_h(\bm{v}_h - \bm{u}, q_h - p;
      \bm{v}_h - \bm{u}_h, q_h - p_h) - B_h(\bm{v}_h
      - \bm{u}_h, q_h - p_h; p - r_h) \\
      &+ \coek(\bm{u} - \bm{u}_h, \bm{v}_h - \bm{u}_h)
      + g_h(\bm{v}_h, \bm{v}_h - \bm{u}_h),
    \end{aligned}
  \end{displaymath}
  for $\forall r_h \in \Qh$. Applying the triangle inequality
  and  the continuity results in Lemma \ref{le_continuity} brings us
  that
  \begin{displaymath}
    \begin{aligned}
      \Anorm{\bm{v}_h - \bm{u}_h} + \cseminorm{q_h - p_h} \leq
      C (\Wnorm{\bm{u}  - \bm{v}_h, p - q_h}  + \Cnorm{p - r_h} +
      k \|\bm{u} - \bm{u}_h\|_{L^2(\Omega)} + \gnorm{\bm{u}}).
    \end{aligned}
  \end{displaymath}
  Then we obtain
  \begin{equation}
    \begin{aligned}
      \Anorm{\bm{u} - \bm{u}_h} + \cseminorm{p - p_h}  \leq  & C \Big(
      \inf_{(\bm{v}_h, q_h) \in \Ker(B_h)}\Wnorm{\bm{u} -
      \bm{v}_h, p - q_h} +  k \|\bm{u} -
       \bm{u}_h \|_{L^2(\Omega)} +  \gnorm{\bm{u}} \Big).
    \end{aligned}
    \label{eq_uulKerinf}
  \end{equation}
  It remains to bound the error $\inf_{(\bm{v}_h, q_h) \in
  \Ker(B_h)} \Wnorm{\bm{u} - \bm{v}_h, p - q_h}$. Fix  
  $(\bm{v}_h, q_h) \in \Vh \times \Qh$, and let $(\bm{w}_h,
  {r}_h)  \in \Vh \times \Qh$ be the solution of the problem 
  \begin{displaymath}
    B_h(\bm{w}_h, {r}_h; t_h) = B_h(\bm{u} - \bm{v}_h, p -q_h; t_h)
    + j_h(p, t_h),
    \quad \forall t_h \in \Qh.
  \end{displaymath}
  The inf-sup condition \eqref{eq_infsup} ensures the existence of the
  solution $(\bm{w}_h, r_h)$, which satisfies 
  \begin{displaymath}
    \Wnorm{\bm{w}_h, r_h} \leq C \sup_{t_h \in \Qh}
    \frac{B_h(\bm{u} - \bm{v}_h, p - q_h; t_h) + j_h(p,
    t_h)}{\Cnorm{t_h}} \leq C  (\Wnorm{\bm{u} - \bm{v}_h, p - q_h} +
    \jnorm{p}). 
  \end{displaymath}
  Note that $B_h(\bm{u}, p; t_h) + j_h(p, t_h) = 0$,  which implies
  $(\bm{w}_h + \bm{v}_h, r_h + q_h) \in \Ker(B_h)$.
  Thus, we obtain 
  \begin{displaymath}
    \begin{aligned}
      \Wnorm{\bm{u} - (\bm{w}_h + \bm{v}_h), p - (r_h + q_h)}
      \leq C  (\Wnorm{\bm{u} - \bm{v}_h, p - p_h} + \jnorm{p}).
    \end{aligned}
  \end{displaymath}
  Combining \eqref{eq_uulKerinf}, we arrive at 
   \begin{equation}
    \begin{aligned}
      \Anorm{\bm{u} - \bm{u}_h} + \cseminorm{p - p_h}  \leq C&\Big(
      \inf_{(\bm{v}_h, q_h) \in  \Vh \times \Qh}
      \Wnorm{\bm{u} - \bm{v}_h, p - q_h}+  k \|\bm{u} -
       \bm{u}_h \|_{L^2(\Omega)} +  \gnorm{\bm{u}} + \jnorm{p} \Big).
    \end{aligned}
    \label{eq_uulinf}
  \end{equation}
  Now we turn to the error $\Cnorm{p - p_h}$. 
  From the inf-sup stability \eqref{eq_infsup}, we obtain 
  \begin{displaymath}
    \begin{aligned}
      &\Cnorm{p_h - q_h} \leq C \sup_{(\bm{v}_h, t_h) \in
      \Vh \times \Qh} \frac{B(\bm{v}_h, t_h;
      q_h - p_h)}{ \Wnorm{\bm{v}_h, t_h}}, \\
    \end{aligned} \quad \forall q_h \in \Qh.
  \end{displaymath}
  The Galerkin orthogonality \eqref{eq_Galerkinorth} gives 
  \begin{displaymath}
    \begin{aligned}
      B(\bm{v}_h, t_h; q_h - p_h) = -& A_h(\bm{u} - \bm{u}_h,
      p - p_h; \bm{v}_h, t_h)  - B_h(\bm{v}_h, t_h; p - q_h) + \coek(\bm{u} - \bm{u}_h,
      \bm{v}_h)  + g_h(\bm{u}_h, \bm{v}_h) \\
      \leq &C  \Wnorm{\bm{v}_h, t_h} (\Anorm{\bm{u} - \bm{u}_h} +
      \cseminorm{ p -
      p_h}) + C \Cnorm{p_h - q_h} \Wnorm{\bm{v}_h, t_h} \\
      &+  k \Wnorm{\bm{v}_h, t_h} \|\bm{u} -
      \bm{u}_h\|_{L^2(\Omega)}  + g_h(\bm{u}_h, \bm{v}_h).
    \end{aligned}
  \end{displaymath}
  For the last term, we have that 
  \begin{displaymath}
    \begin{aligned}
      g_h(\bm{u}_h, \bm{v}_h) & \leq (\gnorm{\bm{u} - \bm{u}_h} +
      \gnorm{\bm{u}}) \gnorm{\bm{v}_h} \leq C ( \Anorm{\bm{u} -
      \bm{u}_h}  + \gnorm{\bm{u}}) \Wnorm{\bm{v}_h,
      t_h}. \\
    \end{aligned}
  \end{displaymath}
  Combing the triangle inequality and all above estimates leads to
  \begin{equation}
    \begin{aligned}
      \Cnorm{p - p_h} \leq C (\Anorm{\bm{u} - \bm{u}_h} + \cseminorm{
      p - p_h}
      +  \gnorm{\bm{u}} +  \inf_{q_h \in \Qh} \Cnorm{p - q_h} +  k
      \|\bm{u} - \bm{u}_h \|_{L^2(\Omega)}).
    \end{aligned}
    \label{eq_pphinf}
  \end{equation}
  From \eqref{eq_uulinf} and the approximation estimates
  \eqref{eq_Anormapp}, \eqref{eq_cnormapp} and \eqref{eq_gnormv}, we
  immediately arrive at the desired estimate \eqref{eq_DGerror}, which
  completes the proof.
\end{proof}
\begin{theorem}
  Under the conditions in Theorem \ref{th_DGerror}, there holds
  \begin{equation}
    \begin{aligned}
      \|\bm{u} - \bm{u}_h\|_{L^2(\Omega)} \leq C_0 h (\Anorm{\bm{u}
      - \bm{u}_h} + \Cnorm{p - p_h}) + C_1 h^{s+1}(
      \|\bm{u} \|_{H^{s+1}(\Omega)} + \| p \|_{H^s(\Omega)}),
    \end{aligned}
    \label{eq_L2error}
  \end{equation}
  where $s = \min(t, m+1)$ and the constants $C_0$, $C_1$ depend on
  $\Cr$.
  \label{th_L2error}
\end{theorem}
\begin{proof}
  We prove the $L^2$ error estimate by the dual argument. 
  Let $(\bm{z}, \psi)$ be the solution of the problem 
  \begin{displaymath}
    \begin{aligned}
      \nabla \times (\coec \nabla \times \bm{z}) - \coek \coed \bm{z}
      + \coed \nabla \psi  &= \coek(\bm{u} - \bm{u}_h), \quad
      \nabla \cdot (\coed \bm{z}) = 0, \quad \text{in } \Omega, \\
      \un \times \bm{z} & = \bm{0}, \quad \psi = 0, \quad \text{on }
      \Gamma.
    \end{aligned}
  \end{displaymath}
  From Theorem \ref{th_regularity}, we know that $\bm{z} \in
  H^2(\Omega)$ and $\psi \in H^1(\Omega)$ with 
  \begin{displaymath}
    \|\bm{z} \|_{H^2(\Omega)} \leq \Cr \coek \|\bm{u} -
    \bm{u}_h\|_{L^2(\Omega)}, \quad \|\psi \|_{H^1(\Omega)} \leq
    C \coek \| \bm{u} - \bm{u}_h\|_{L^2(\Omega)}.
  \end{displaymath}
  Let $(\bm{z}_h, \psi_h) \in \Vh \times \Qh$ be their interpolation
  functions. Applying integration by parts and 
  the Galerkin orthogonality \eqref{eq_Galerkinorth}, we see that
  \begin{displaymath}
    \begin{aligned}
      \coek \| \bm{u} - &\bm{u}_h \|_{L^2(\Omega)}^2 ={A}_h(\bm{z},
       \psi; \bm{u} - \bm{u}_h, p - p_h) + {B}_h(\bm{u} -
      \bm{u}_h, p - p_h; \psi) - \coek (\bm{u} - \bm{u}_h,
      \bm{z}) \\ 
      &= {A}_h( \bm{u} - \bm{u}_h, p - p_h; \bm{z} -
      \bm{z}_h, \psi - \psi_h)  - {B}_h(\bm{z}_h, \psi_h; p - p_h) +
      g_h(\bm{u}_h, \bm{z}_h) \\ 
      & \quad+ {B}_h(\bm{u} - \bm{u}_h, p - p_h; \psi - \psi_h) -
      \coek (\bm{u} - \bm{u}_h, \bm{z} - \bm{z}_h) - j_h(p, \psi_h) .
    \end{aligned}
  \end{displaymath}
  From Lemma \ref{le_continuity} and the approximation property
  \eqref{eq_Anormapp}, there hold
  \begin{displaymath}
    \begin{aligned}
      A_h( \bm{u} - \bm{u}_h, p - p_h; \bm{z} - \bm{z}_h,
      \psi - \psi_h) &\leq C h\Wnorm{\bm{u} - \bm{u}_h, p - p_h} ( \|
      \bm{z} \|_{H^2(\Omega)} + \| \psi\|_{H^1(\Omega)}), \\
      {B}_h(\bm{u} - \bm{u}_h, p - p_h; \psi - \psi_h) &\leq
      Ch \Wnorm{\bm{u} - \bm{u}_h, p - p_h} \|\psi\|_{H^1(\Omega)}.
    \end{aligned}
  \end{displaymath}
  Note that $B_h(\bm{z}, \psi; p - p_h) = -j_h(\psi, p - p_h)$. We  can
  further deduce that 
  \begin{displaymath}
    \begin{aligned}
      B_h(\bm{z}_h, \psi_h; p - p_h) &= B_h(\bm{z}_h - \bm{z}, \psi_h -
      \psi; p - p_h) - j_h(\psi, p - p_h) \\
      & \leq  Ch(\|\bm{z} \|_{H^2(\Omega)}  + \| \psi
      \|_{H^1(\Omega)})\Cnorm{p - p_h}.
    \end{aligned}
  \end{displaymath}
  From the approximation estimates \eqref{eq_Anormapp} and
  \eqref{eq_cnormapp}, we show that 
  \begin{displaymath}
    \begin{aligned}
      g_h(\bm{u}_h, \bm{z}_h) &= g_h(\bm{u}_h - \bm{u}, \bm{z}_h) +
      g_h(\bm{u}, \bm{z}_h) \leq C \Anorm{\bm{u} - \bm{u}_h}\gnorm{\bm{z}_h} +
      \gnorm{\bm{u}}\gnorm{\bm{z}_h} \\
      & \leq Ch \Anorm{\bm{u} - \bm{u}_h} \|\bm{z}
      \|_{H^2(\Omega)} + Ch^{s+1} \|\bm{u} \|_{H^{s+1}(\Omega)}
      \|\bm{z} \|_{H^2(\Omega)},
    \end{aligned}
  \end{displaymath}
  and 
  \begin{displaymath}
    j_h(p, \psi_h) \leq \jnorm{p} \jnorm{\psi_h} \leq Ch^{s+1} \| p
    \|_{H^s(\Omega)} \| \psi \|_{H^1(\Omega)}.
  \end{displaymath}
  Putting together all the above estimates yields the error estimate
  \eqref{eq_L2error}, which completes the proof.
\end{proof}
Combining   Theorem \ref{th_DGerror} and Theorem
\ref{th_L2error} leads to the desired error estimates in Theorem
\ref{th_errorestimate}.  Particularly, we point out that for the
divergence-free source term $\bm{j}$, the solution $p$ is just zero,
and this case is also often encountered in practice
\cite{Perugia2002stabilized}.  For such a problem, $p$ can be
approximated by piecewise constant spaces, i.e. 
we can use the pair of spaces $\Vh \times Q_h^0(r \geq 1)$ in the
numerical scheme. For the exact solution $(\bm{u}, p)$ with $\bm{u}
\in H^{t + 1}(\Omega)^d$ and $p = 0$, by estimates \eqref{eq_uulinf}
and \eqref{eq_pphinf}, the numerical solution $(\bm{u}_h, p_h)$ from
the discrete problem \eqref{eq_mixedform} has the following error
estimates:
\begin{equation}
  \begin{aligned}
    \Anorm{\bm{u} - \bm{u}_h} + \Cnorm{p - p_h} & \leq C_0  h^{s}
    \|\bm{u} \|_{H^{s+1}(\Omega)}, \\
    \| \bm{u} - \bm{u}_h \|_{L^2(\Omega)} & \leq C_1 h^{s+1} \|\bm{u}
    \|_{H^{s+1}(\Omega)}, \\
  \end{aligned}
  \label{eq_p0error}
\end{equation}
where $C_1$ depends on $\Cr$ and $s = \min(t, r)$, provided by $h \leq
h_1$.

\section{Numerical Results}
\label{sec_numericalresults}
In this section, a number of numerical tests in two and three
dimensions are presented to show the numerical performance of
the unfitted method. For all tests, the source term $\bm{j}$ and the
boundary data $\bm{g}$ are taken accordingly from the exact solution
$\bm{u}$. We assume that the coefficients $\mu_r$ and $\coed$ correspond to
the case of the vacuum, i.e. $\mu_r = \coed = 1$,  and we also take $k =
1$. The curved boundary $\Gamma$ in each case is described by a level
set function $\phi$. 

\subsection{2D Examples} 
The spaces $\Vh$ and $\Qh$ for the two-dimensional case are selected
to be the discontinuous piecewise polynomial spaces. 
The penalty parameter $\alpha$ is taken as
$3(m+1)^2 + 15$. We refer to \cite{Sarmany2010optimal,
Prenter2018note} for some discussions on the choice of the
parameter. 

\def\Vh{\bmr{V}_h^{m+1}}

\paragraph{\textbf{Example 1.}}
In the first example, we solve the Maxwell problem \eqref{eq_Maxwell}
defined in a circle centered at the origin with radius $r = 0.7$. 
The corresponding level set function $\phi(x, y)$ is 
\begin{displaymath}
  \phi(x, y) = x^2 + y^2 - r^2, \quad \forall(x, y) \in \mb{R}^2,
  \quad r = 0.7.
\end{displaymath}
The domain $\Omega = \{ (x, y) \in \mb{R}^2 \ | \ \phi(x, y) < 0 \}$
and the background mesh $\MTh^*$ is taken to cover the squared domain
$(-1, 1)^2$ (cf. Figure \ref{fig_ex1mesh}), with the mesh size $h =
1/3, 1/6, 1/12, 1/24$. The exact solution is chosen as
\begin{displaymath}
  \bm{u}(x, y) = \begin{bmatrix} 
    \cos(\pi x)\sin(\pi y) \\  -\sin(\pi x)\cos(\pi y)) \\
  \end{bmatrix},
  \quad p(x, y) = x^2 + y^2 - r^2.
\end{displaymath}
We solve this problem with the pair of spaces $\Vh
\times \Qh (0 \leq m \leq 2)$. The numerical errors are reported in
Table \ref{tab_example1}. It can be observed that for $\bm{u}$ the
errors under both the energy norm and the $L^2$ norm approach zero at
the optimal convergence speed $O(h^{m+1})$ and $O(h^{m+2})$,
respectively, and the error $\Cnorm{p - p_h}$ also converges to zero
with the optimal rate $O(h^{m+1})$. The numerical results are in
perfect agreement with the theoretical estimate \eqref{eq_mixederror}. 

\begin{figure}[htp]
  \centering
  \begin{minipage}[t]{0.36\textwidth}
    \centering
    \begin{tikzpicture}[scale=1.7]
      \centering
      \node at (0, 0) {$\Omega$};
      \draw[thick, red] (0, 0) circle [radius = 0.7];
      \draw[thick, black] (-1, -1) rectangle (1, 1);
    \end{tikzpicture}
  \end{minipage}
  \begin{minipage}[t]{0.36\textwidth}
    \centering
    \begin{tikzpicture}[scale=1.7]
      \centering
      \input{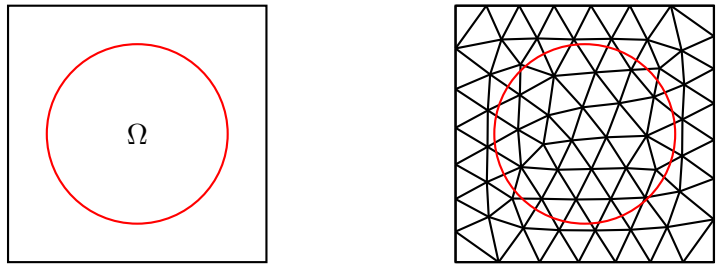}
      \draw[thick, red] (0, 0) circle [radius = 0.7];
      \draw[thick, black] (-1, -1) rectangle (1, 1);
    \end{tikzpicture}
  \end{minipage}
  \caption{The curved domain $\Omega$ and the partition $\MTh^*$ of
  Example 1 in two dimensions.}
  \label{fig_ex1mesh}
\end{figure}

\begin{table}
  \centering
  \renewcommand\arraystretch{1.3}
  \scalebox{0.8}{
  \begin{tabular}{p{0.3cm} | p{2.6cm} | p{1.6cm} | p{1.6cm} |
    p{1.6cm} | p{1.6cm} | p{1cm} }
    \hline\hline
    $m$ & $h$ & 1/3 & 1/6 & 1/12 & 1/24 & order \\
    \hline
    \multirow{3}{*}{$1$} & $\| \bm{u}-\bm{u}_h \|_{L^2(\Omega)}$
    & 1.528e-1 & 3.433e-2 & 7.307e-3 & 1.850e-3 & 1.99 \\
    \cline{2-7}
    & $ \Anorm{\bm{u}-\bm{u}_h} $
    & 1.476e-0 & 6.662e-1 & 3.158e-1 & 1.585e-1 & 0.99 \\
    \cline{2-7}
    & $ \Cnorm{p-p_h} $ & 1.553e-1 & 6.901e-2 & 2.741e-2 & 1.280e-2 &
    1.10 \\
    \hline
    \multirow{3}{*}{$2$} & $\| \bm{u}-\bm{u}_h \|_{L^2(\Omega)}$
    & 3.402e-2 & 1.607e-3 & 1.871e-4 & 1.946e-5 & 3.26 \\
    \cline{2-7}
    & $ \Anorm{\bm{u}-\bm{u}_h} $
    & 4.067e-1 & 6.509e-1 & 1.491e-2 & 3.293e-3 & 2.18 \\
    \cline{2-7}
    & $ \Cnorm{p-p_h} $ & 3.033e-2 & 5.113e-3 & 1.080e-3 & 2.103e-4 &
    2.36 \\
    \hline
    \multirow{3}{*}{$3$} & $\| \bm{u}-\bm{u}_h \|_{L^2(\Omega)}$
    & 3.553e-3 & 1.436e-4 & 6.665e-6 & 3.489e-7 & 4.26 \\
    \cline{2-7}
    & $ \Anorm{\bm{u}-\bm{u}_h} $
    & 7.159e-2 & 6.542e-3 & 6.096e-4 & 5.481e-5 & 3.48 \\
    \cline{2-7}
    & $ \Cnorm{p-p_h} $ & 2.061e-3 & 2.592e-4 & 2.595e-5 & 2.918e-6 &
    3.15 \\
    \hline\hline
  \end{tabular}
  }
  \caption{The numerical errors of Example 1 in 2D.}
  \label{tab_example1}
\end{table}

\paragraph{\textbf{Example 2.}}
In this example, we consider the problem in a star-shaped domain
\cite{Massing2019stabilized} (see Figure \ref{fig_example2}), where
the boundary $\Gamma$ is governed by the following level set
function in the polar coordinate $(r, \theta)$:
\begin{displaymath}
  \phi(r, \theta) = r - \frac{1}{2} - \frac{\sin(5 \theta)}{7}.
\end{displaymath}
The problem is solved with a family of triangular meshes on the domain
$\Omega^* = (-1, 1)^2$ with the mesh size $h = 1/6$, $1/12$.  $1/24$.
$1/48$. The analytical solution is given by
\begin{displaymath}
  \bm{u}(x, y) = \begin{bmatrix}
    -e^{x}(y \cos(y) + \sin(y)) \\
    e^x y \sin(y) \\
  \end{bmatrix}, \quad p = 0.
\end{displaymath}
Since $p = 0$ in this case, the problem is approximated by the spaces
$\Vh \times Q_h^0$ with the accuracy $0 \leq m \leq 2$, as suggested
in Section \ref{sec_erroraux}.  The convergence history is shown in
Table \ref{tab_example2}. The errors $\Anorm{\bm{u} - \bm{u}_h}$ and
$\|\bm{u} - \bm{u}_h \|_{L^2(\Omega)}$ decrease to zero at the optimal
rates $O(h^{m+1})$ and $O(h^{m+2})$, respectively, which are well
consistent with the estimate \eqref{eq_p0error}. Although $Q_h^0$
is the piecewise constant space, we still observe that the error
$\Cnorm{p - p_h}$ tends to zero at the speed $O(h^{m+1})$, which
validates the estimate \eqref{eq_p0error}.

\begin{figure}[htp]
  \centering
  \begin{minipage}[t]{0.35\textwidth}
    \centering
    \begin{tikzpicture}[scale=1.7]
      \centering
      \node at (0, 0) {$\Omega$};
      \draw[thick, black] (-1, -1) rectangle (1, 1);
      \draw[thick, domain=0:360, red, samples=120] plot (\x:{(0.5 -
      sin(\x*5)/7)*1.02});
    \end{tikzpicture}
  \end{minipage}
  \begin{minipage}[t]{0.35\textwidth}
    \centering
    \begin{tikzpicture}[scale=1.7]
      \centering
      \input{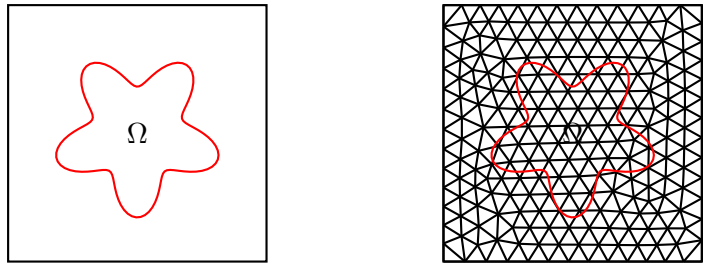}
      \node at (0, 0) {$\Omega$};
      \draw[thick, black] (-1, -1) rectangle (1, 1);
      \draw[thick, domain=0:360, red, samples=120] plot (\x:{(0.5 -
      sin(\x*5)/7)*1.02});
    \end{tikzpicture}
  \end{minipage}
  \caption{The curved domain $\Omega$ and the partition $\MTh^*$ of
  Example 2 in 2D.}
  \label{fig_example2}
\end{figure}

\begin{table}
  \centering
  \renewcommand\arraystretch{1.3}
  \scalebox{0.8}{
  \begin{tabular}{p{0.3cm} | p{2.6cm} | p{1.6cm} | p{1.6cm} |
    p{1.6cm} | p{1.6cm} | p{1cm} }
    \hline\hline
    $m$ & $h$ & 1/3 & 1/6 & 1/12 & 1/24 & order \\
    \hline
    \multirow{3}{*}{$1$} & $\| \bm{u}-\bm{u}_h \|_{L^2(\Omega)}$
    & 1.069e-2 & 1.742e-3 & 4.696e-4 & 1.098e-4 & 2.09 \\
    \cline{2-7}
    & $ \Anorm{\bm{u}-\bm{u}_h} $
    & 1.658e-1 & 7.645e-2 & 3.667e-2 & 1.818e-2 & 1.02 \\
    \cline{2-7}
    & $ \Cnorm{p-p_h} $ & 1.665e-2 & 5.650e-3 & 2.800e-3 & 1.185e-3 &
    1.23 \\
    \hline
    \multirow{3}{*}{$2$} & $\| \bm{u}-\bm{u}_h \|_{L^2(\Omega)}$
    & 1.495e-4 & 1.538e-5 & 1.328e-6 & 1.480e-7 & 3.16 \\
    \cline{2-7}
    & $ \Anorm{\bm{u}-\bm{u}_h} $
    & 4.942e-3 & 1.068e-3 & 2.412e-4 & 5.783e-5 & 2.06 \\
    \cline{2-7}
    & $ \Cnorm{p-p_h} $ & 2.308e-4 & 2.050e-5 & 2.248e-6 & 3.476e-7 &
    2.68 \\
    \hline
    \multirow{3}{*}{$3$} & $\| \bm{u}-\bm{u}_h \|_{L^2(\Omega)}$
    & 4.120e-6 & 1.373e-7 & 4.388e-9 & 2.344e-10 & 4.22 \\
    \cline{2-7}
    & $ \Anorm{\bm{u}-\bm{u}_h} $
    & 1.388e-4 & 1.075e-5 & 1.073e-6 & 1.233e-7 & 3.12 \\
    \cline{2-7}
    & $ \Cnorm{p-p_h} $ & 1.763e-6 & 1.603e-7 & 4.871e-9 & 5.192e-10 &
    3.23 \\
    \hline\hline
  \end{tabular}
  }
  \caption{The numerical errors of Example 2 in 2D.}
  \label{tab_example2}
\end{table}

\subsection{3D Examples}
The spaces $\Vh$ and $\Qh$ in three dimensions are selected to be 
$C^0$ finite element spaces.
The penalty parameter $\alpha$ is taken as $3m^2 + 25$.

\paragraph{\textbf{Example  3.}}
We first test a three-dimensional example by solving the problem
defined in a sphere, which is centered at $(0.5, 0.5, 0.5)$ with  
radius $r = 0.35$, i.e. the corresponding level set function is 
\begin{displaymath}
  \phi(x, y, z) = (x - 0.5)^2 + (y-0.5)^2 + (z - 0.5)^2 - r^2, \quad r
  = 0.35.
\end{displaymath}
We employ a series of tetrahedral meshes on the cubic domain $(0,
1)^3$ to solve this problem; see Figure \ref{fig_3dex1domain}.  The
exact solution takes the form 
\begin{displaymath}
  \bm{u} = \begin{bmatrix}
    \sin(\pi y) \sin(\pi z) \\
    \sin(\pi x) \sin(\pi z) \\
    \sin(\pi x) \sin(\pi y) \\
  \end{bmatrix}, \quad p = (x-0.5)^2 + (y-0.5)^2 + (z-0.5)^2 - 0.35^2.
\end{displaymath}
The numerical results are gathered in Table \ref{tab_example3d1}. One
can observe that the numerical solution of the unfitted method still
has the optimal rates under all error measurements in three
dimensions, which are clearly in accordance with the theoretical
estimates. 

\begin{figure}[htb]
  \centering
  \includegraphics[width=0.22\textwidth, height=0.22\textwidth]{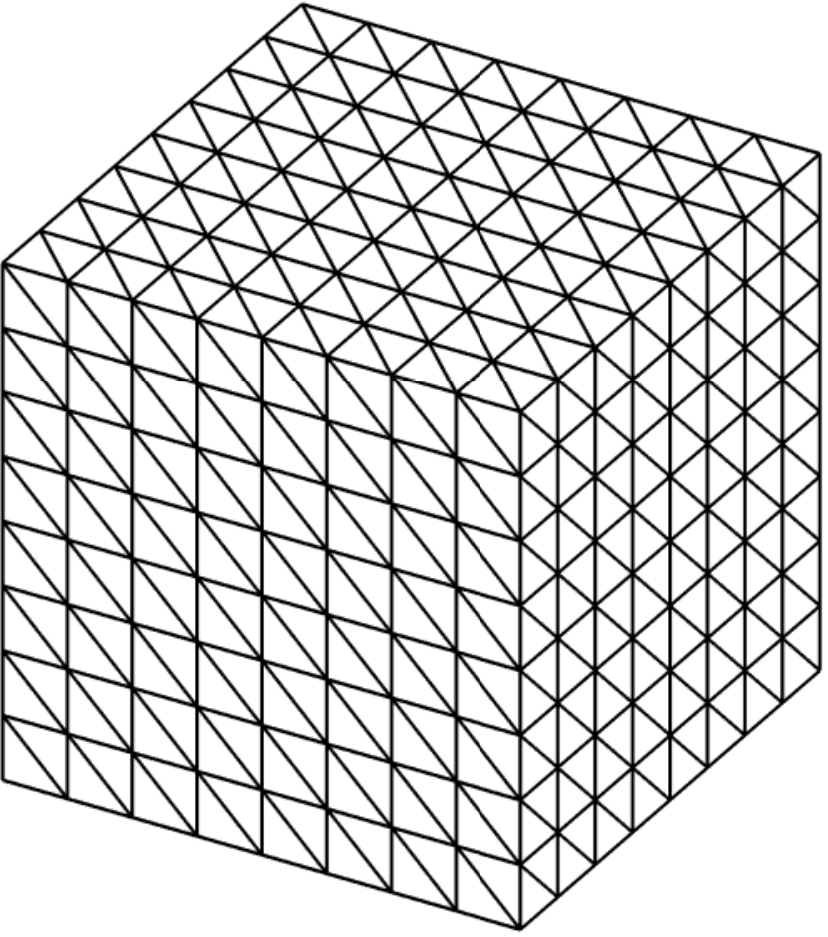}
  \hspace{50pt}
  \includegraphics[width=0.22\textwidth, height=0.22\textwidth]{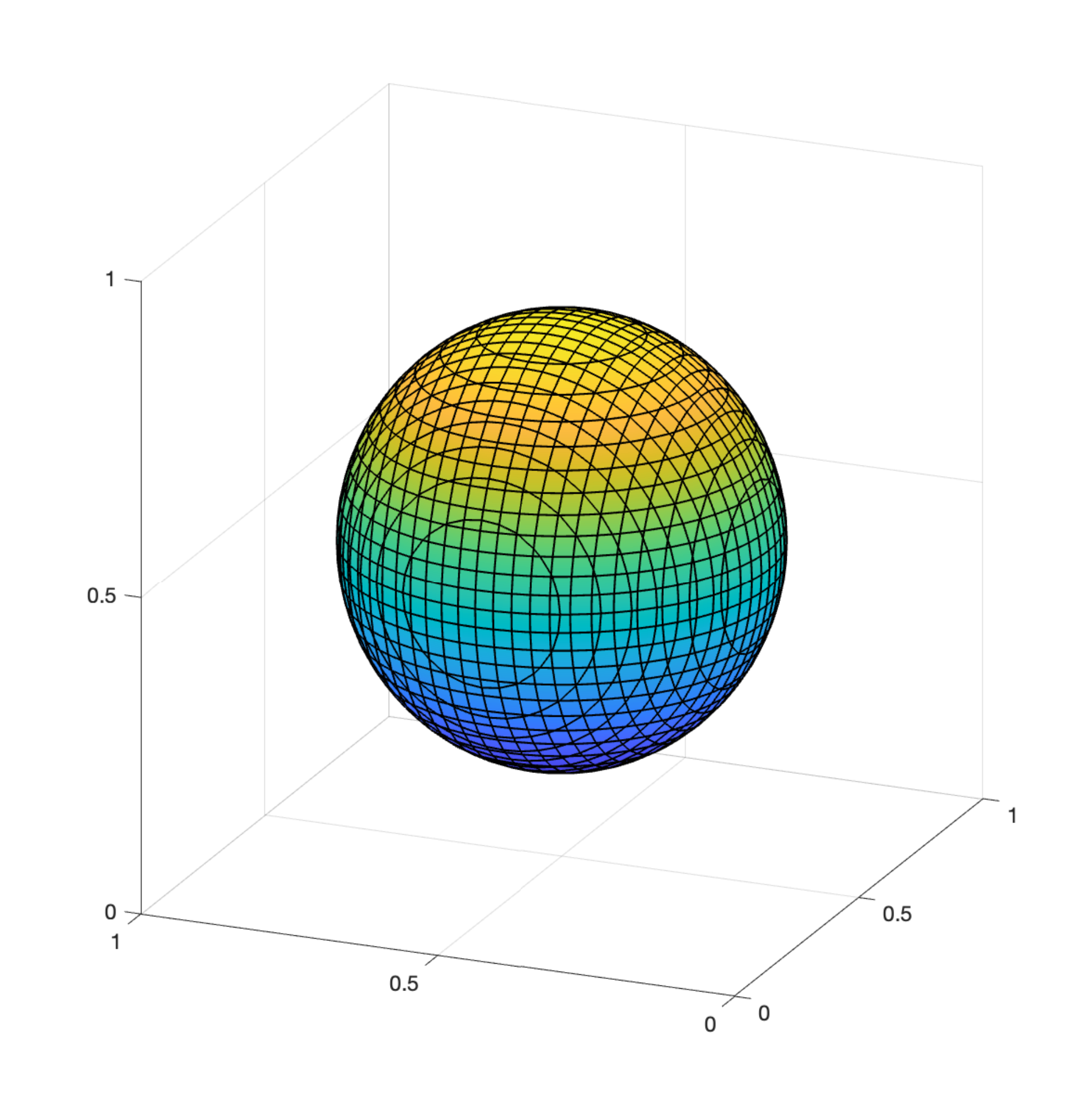}
  \caption{The spherical domain $\Omega$ and the partition $\MTh^*$ of
  Example 3 in 3D.}
  \label{fig_3dex1domain}
\end{figure}

\begin{table}
  \centering
  \renewcommand\arraystretch{1.3}
  \scalebox{0.8}{
  \begin{tabular}{p{0.3cm} | p{2.6cm} | p{1.6cm} | p{1.6cm} |
    p{1.6cm} | p{1.6cm} | p{1cm} }
    \hline\hline
    $m$ & $h$ & 1/4 & 1/8 & 1/16 & 1/32 & order \\
    \hline
    \multirow{3}{*}{$1$} & $\| \bm{u}-\bm{u}_h \|_{L^2(\Omega)}$
    & 6.169e-2 & 1.328e-2 & 2.344e-3 & 5.022e-4 & 2.22 \\
    \cline{2-7}
    & $ \Anorm{\bm{u}-\bm{u}_h} $
    & 6.916e-1 & 2.657e-1 & 1.235e-1 & 5.992e-2 & 1.03 \\
    \cline{2-7}
    & $ \Cnorm{p-p_h} $ & 2.046e-2 & 1.255e-2 & 6.058e-3 & 2.320e-3 &
    1.38 \\
    \hline
    \multirow{3}{*}{$2$} & $\| \bm{u}-\bm{u}_h \|_{L^2(\Omega)}$
    & 2.102e-3 & 4.746e-4 & 4.042e-5 & 3.958e-6 & 3.35 \\
    \cline{2-7}
    & $ \Anorm{\bm{u}-\bm{u}_h} $
    & 4.413e-2 & 1.928e-2 & 3.830e-3 & 8.768e-4 & 2.13 \\
    \cline{2-7}
    & $ \Cnorm{p-p_h} $ & 9.755e-3 & 2.650e-3 & 2.739e-4 & 6.039e-5 &
    2.18 \\
    \hline\hline
  \end{tabular}
  }
  \caption{The numerical errors of Example 3 in 3D.}
  \label{tab_example3d1}
\end{table}

\paragraph{\textbf{Example 4.}}
In the last example, we consider the case that the boundary of the
domain $\Gamma$ is a smooth molecular surface of two atoms (see Figure
\ref{fig_3dex2domain}), whose level set function reads
\cite{Wei2018spatially, Li2018interface}:
\begin{displaymath}
  \phi(x, y, z) = \left( (2.5(x - 0.5))^2 + (4(y-0.5))^2 + (2.5(z -
  0.5))^2 + 0.6 \right)^2 - 3.5(4(y-0.5))^2 - 0.6.
\end{displaymath}
We solve the problem \eqref{eq_Maxwell} with the exact solution 
\begin{displaymath}
  \bm{u} = \begin{bmatrix}
    \cos(x) \sin(y) e^{2z} \\
    \sin(x) \cos(y) e^{2z} \\
    \sin(x) \sin(y) e^{2z} \\
  \end{bmatrix}, \quad p = 0.
\end{displaymath}
For this case, we also adopt the approximation spaces $\Vh \times
Q_h^0(m = 0, 1)$ in the numerical scheme, and use a sequence of
tetrahedral meshes with $h = 1/4, 1/8, 1/16, 1/32$. The
results are collected in Table \ref{tab_example3d2}. The numerically
detected convergence rates again confirm the theoretical predictions.

\begin{figure}[htb]
  \centering
  \includegraphics[width=0.25\textwidth]{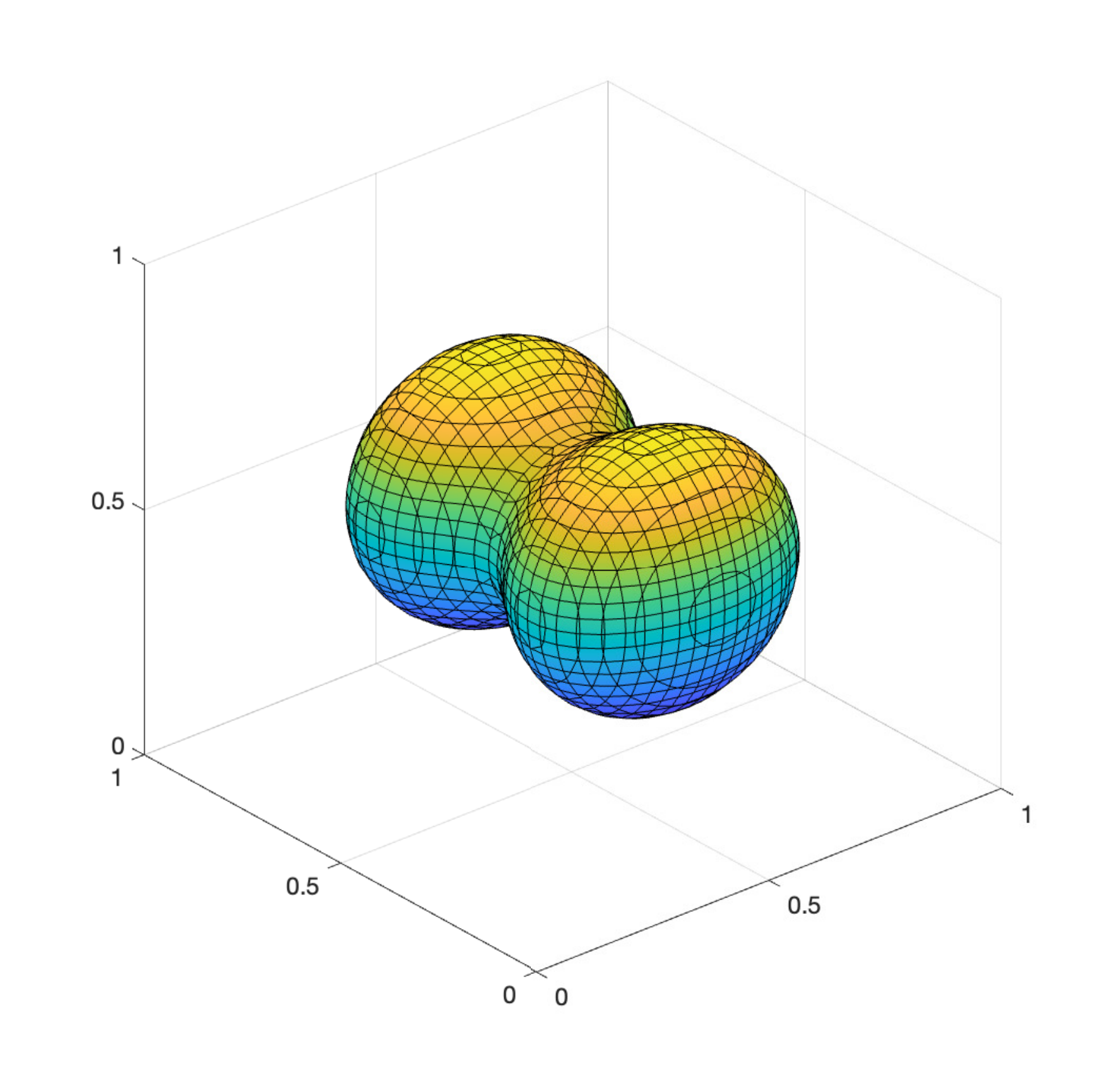}
  \hspace{50pt}
  \includegraphics[width=0.23\textwidth]{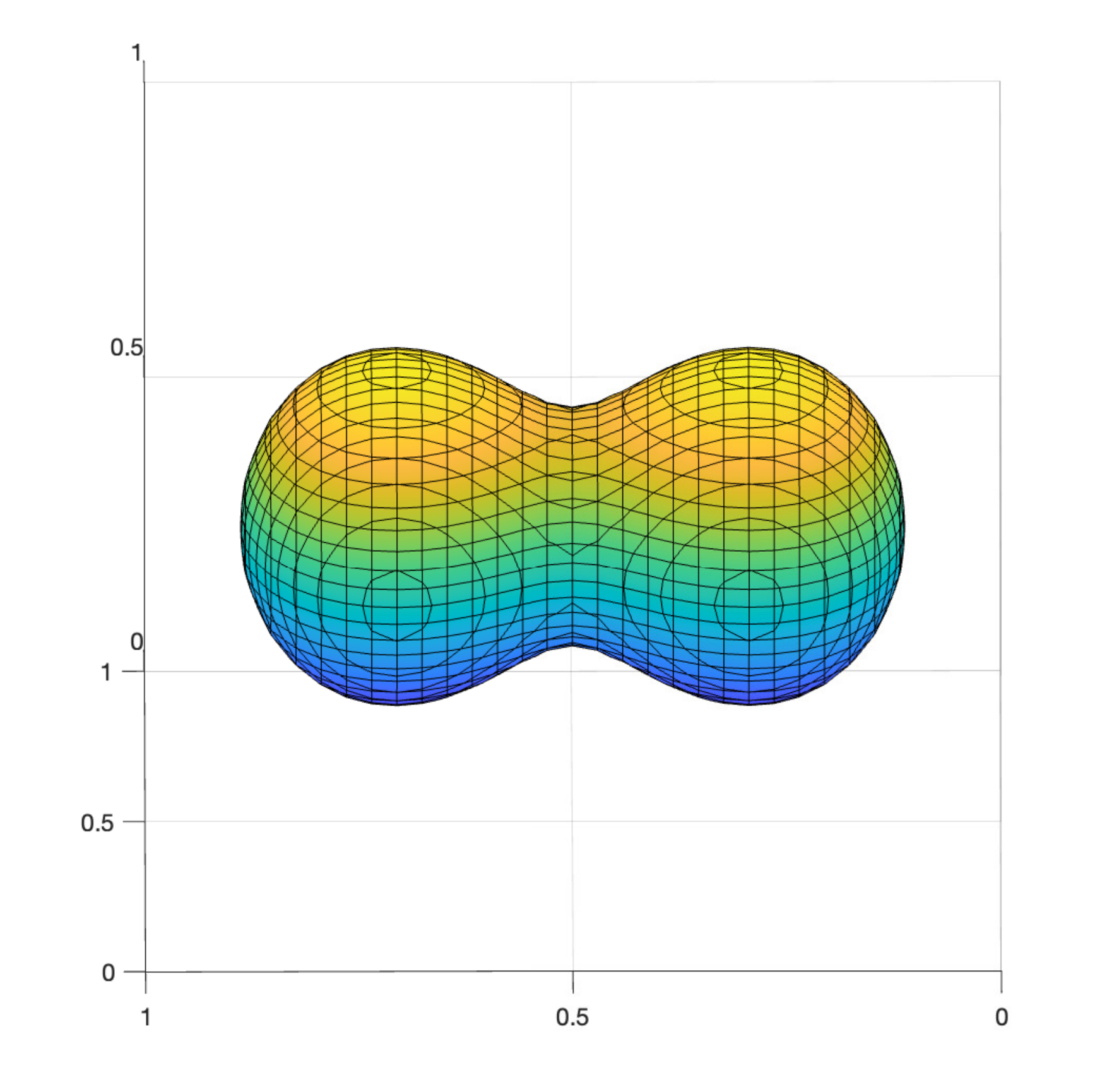}
  \caption{The curved domain $\Omega$ of Example 4 in 3D.}
  \label{fig_3dex2domain}
\end{figure}

\begin{table}
  \centering
  \renewcommand\arraystretch{1.3}
  \scalebox{0.8}{
  \begin{tabular}{p{0.3cm} | p{2.6cm} | p{1.6cm} | p{1.6cm} |
    p{1.6cm} | p{1.6cm} | p{1cm} }
    \hline\hline
    $m$ & $h$ & 1/4 & 1/8 & 1/16 & 1/32 & order \\
    \hline
    \multirow{3}{*}{$1$} & $\| \bm{u}-\bm{u}_h \|_{L^2(\Omega)}$
    & 1.701e-2 & 5.122e-3 & 1.353e-3 & 3.317e-4 & 2.03 \\
    \cline{2-7}
    & $ \Anorm{\bm{u}-\bm{u}_h} $
    & 4.298e-1 & 2.355e-1 & 1.212e-1 & 6.121e-2 & 0.99 \\
    \cline{2-7}
    & $ \Cnorm{p-p_h} $ & 2.221e-2 & 2.055e-2 & 1.053e-2 & 4.578e-3 &
    1.20 \\
    \hline
    \multirow{3}{*}{$2$} & $\| \bm{u}-\bm{u}_h \|_{L^2(\Omega)}$
    & 3.682e-3 & 1.262e-4 & 1.576e-5 & 1.890e-6 & 3.06 \\
    \cline{2-7}
    & $ \Anorm{\bm{u}-\bm{u}_h} $
    & 9.188e-2 & 8.671e-2 & 2.228e-3 & 5.383e-4 & 2.05 \\
    \cline{2-7}
    & $ \Cnorm{p-p_h} $ & 1.308e-2 & 3.703e-4 & 1.139e-4 & 2.966e-5 &
    1.93 \\
    \hline\hline
  \end{tabular}
  }
  \caption{The numerical errors of Example 4 in 3D.}
  \label{tab_example3d2}
\end{table}

\section{Conclusion}
\label{sec_conclusion}
In this paper, we have developed an unfitted finite element method for
the time-harmonic Maxwell equations on a smooth domain, based on a
local extension operator and a ghost penalty technique.  The unfitted
mixed  interior penalty scheme allows the curved  boundary  to
intersect the background mesh arbitrarily, and is of optimal
convergence rates under both the energy norm and the $L^2$ norm for
all variables.  A number of numerical results  have confirmed the
theoretical predictions.


\newcommand\Henorm[1]{|#1|_{e}}
\newcommand\Hcnorm[1]{|#1|_{\circ}}

\def\Vo{\bmr{V}_{h}^{1}}

\begin{appendix}
  \section{}
  \label{sec_app_proof}
  The proof of Lemma \ref{le_infsupTH} follows from the ideas in
  \cite{Guzman2018infsup}.  
  Define the scaled seminorms
  \begin{displaymath}
    \Hcnorm{q_h}^2 := \sum_{K \in \MThc} h_K^2 \|\nabla q_h
    \|_{L^2(K)}^2, \quad \forall q_h \in \Qhcc
  \end{displaymath}
  \begin{displaymath}
    \Henorm{q_h}^2 := \sum_{K \in \MTh} h_K^2 \|\nabla q_h
    \|_{L^2(K)}^2 + \sum_{f \in \MFhI} h_f \| \jump{q_h}
    \|_{L^2(f)}^2, \quad \forall q_h \in \Qh.
  \end{displaymath}
  Let us first recall some existing results   \cite[Lemma 3 and Lemma
  4]{Guzman2018infsup}:
  \begin{lemma}
    For any $q_h \in \Qhcc$, there exists $\wt{q}_h \in \Qh$ such
    that $\wt{q}_h|_{\Ohc} = q_h$ and 
    \begin{equation}
      \Henorm{\wt{q}_h} \leq C \Hcnorm{q_h}, \quad \|\wt{q}_h
      \|_{L^2(\Oho)} \leq C \| q_h \|_{L^2(\Ohc)}.
    \label{eq_ap_Ehq}
    \end{equation}
    \label{le_ap_Ehq}
  \end{lemma}
  \begin{lemma}
    For any $\bm{v}_h \in \Vo \cap H^1(\Omega_h)^d$, there
    exists a unique decomposition $\bm{v}_h =\bm{v}_{h, 1} +
    \bm{v}_{h, 2}$ such that $\bm{v}_{h, 1} \in
    H_0^1(\Ohc)^d$, $\supp(\bm{v}_{h, 1}) = \Ohc$, $
    \bm{v}_{h, 2} \in  \Vo \cap H^1(\Oh)^d$, and
    \begin{displaymath}
      \sum_{K \in \MTh} h_K^{-2} \|\bm{v}_{h, 2} \|_{L^2(K)}^2 \leq
      \sum_{K \in \MThG} h_K^{-2} \|\bm{v}_h \|_{L^2(K)}^2.
    \end{displaymath}
    \label{le_ap_vhdecompose}
  \end{lemma}
  We state the following inf-sup stability property:
  \begin{lemma}
    For $m \geq 1$, there holds
    \begin{equation}
      \sup_{\bm{v}_h \in \Vhcc \cap H_0^1(\Ohc)^d}
      \frac{\int_{\Ohc}\nabla \cdot (\coed \bm{v}_h) q_h
      \d{x}}{\|\bm{v}_h \|_{H^1(\Ohc)}} \geq C\Hcnorm{q_h}, \quad
      \forall q_h \in \Qhcc.
      \label{eq_ap_infsup}
    \end{equation}
    \label{le_ap_infsup}
  \end{lemma}
  \begin{proof}
  This result is a modification of \cite[Assumption
  3]{Guzman2018infsup}. We assume that every element $K \in \MThc$
  has an interior vertex. We denote by $\MEhcI$ the set of all
  interior edges in $\MThc$ ($\MEhcI$ = $\MFhcI$ in two dimensions).
  Let   $\bm{x}_e$ be the midpoint of the edge $e$. 
  Then let $\phi_e$ be the Lagrange basis function of the second-order
  $C^0$ space corresponding to the point $\bm{x}_e$.
  For any $e \in
  \MEhcI$, we have that ${\phi}_e \in H_0^1(\Ohc)$ and
  ${\phi}_e(\bm{x}_e) = 1$ and ${\phi}_e(\bm{x}) \geq 0 (\forall
  \bm{x} \in \Ohc)$. We define $ \bm{v}_h = - \sum_{e \in \MEhcI}
  h_e^2 {\phi}_e(\bm{x}) (\bm{t}_e \cdot \nabla q_h(\bm{x}))
  \bm{t}_e$,
  where $\bm{t}_e$ is the unit tangential vector on $e$. Since $q_h$
  is continuous on the domain $\Ohc$, there
  holds $\bm{v}_h \in H_0^1(\Omega_h^{\circ})^d$. Thus, applying the
  integration by parts yields that
  \begin{displaymath}
    \begin{aligned}
      \int_{\Ohc} &\nabla \cdot (\coed \bm{v}_h) q_h
      \d{x} = \sum_{e \in \MEhcI} h_e^2 \int_{w(e)} \coed \phi_e
      |\bm{t}_e \cdot \nabla q_h|^2 \d{x} \geq C \sum_{e \in
      \MEhcI} h_e^2 \int_{w(e)} \phi_e |\bm{t}_e \cdot \nabla q_h|^2
      \d{x}, \\
    \end{aligned}
  \end{displaymath}
  where $w(e)$ is the set of elements that have the edge $e$. From the
  inverse inequality, we have that
  \begin{displaymath}
    \begin{aligned}
      \sum_{e \in \MEhcI} h_e^2 \int_{w(e)} \phi_e |\bm{t}_e \cdot
      \nabla q_h|^2 \d{x} 
      \geq C \sum_{K \in \MThc}  h_K^2 \|\nabla q_h \|_{L^2(K)}^2 =
      C \Hcnorm{q_h}^2.
    \end{aligned}
  \end{displaymath}
  In
  addition, it is trivial to check $\|\bm{v}_h
  \|_{H^1(\Omega_h^\circ)} \leq C \Hcnorm{q_h}$, which gives the
  estimate \eqref{eq_ap_infsup}. 
\end{proof}

  
Now let us verify the estimate \eqref{eq_infsupTH}. Given $q_h
\in \Qhcc$, we let $\wt{q}_h \in \Qh$ be  the corresponding piecewise
polynomial function in Lemma \ref{le_ap_Ehq}. As \eqref{eq_coedv}, we
let $\bm{v} \in H_0^1(\Omega)$ be the solution to  $ \nabla \cdot
(\coed \bm{v}) = \wt{q}_h$ in $\Omega$,
and extend $\bm{v}$ to $\Oh$ by zero, and let $\bm{v}_h \in
\bmr{V}_{h}^1$ be its Scott-Zhang interpolant. Decomposing $\bm{v}_h =
\bm{v}_{h, 1} + \bm{v}_{h, 2}$ as in Lemma \ref{le_ap_vhdecompose},
  we further have that 
\begin{displaymath}
  \| \wt{q}_h \|_{L^2(\Omega)}^2 = (\nabla \cdot (\coed \bm{v}_{h, 1}),
  \wt{q}_h)_{L^2(\Oh)} + (\nabla \cdot (\coed \bm{v}_{h, 2}),
  \wt{q}_h)_{L^2(\Oh)} + (\nabla \cdot (\coed(\bm{v} - \bm{v}_h)),
  \wt{q}_h)_{L^2(\Oh)}.
\end{displaymath}
Applying  integration by parts, the approximation property of
$\bm{v}_h$, and  \eqref{eq_ap_Ehq}, we get
\begin{displaymath}
  (\nabla \cdot (\coed (\bm{v} - \bm{v}_h)), \wt{q}_h)_{L^2(\Oh)}
  \leq C \|\wt{q}_h \|_{L^2(\Omega)} \Henorm{\wt{q}_h} \leq C
  \|q_h \|_{L^2(\Ohc)} \Hcnorm{q_h}.
\end{displaymath}
Using the fact that $\bm{v}_h = \bm{0}$ on $\partial \Oh$ and Lemma
\ref{le_ap_vhdecompose}, we have that 
\begin{displaymath}
  \begin{aligned}
    (\nabla \cdot &(\coed \bm{v}_{h, 2}), \wt{q}_h)_{L^2(\Oh)} =
    -(\coed  \bm{v}_{h, 2}, \nabla \wt{q}_h)_{L^2(\Oh)} + \sum_{f
    \in \MFhI} \int_f (\coed \bm{v}_{h, 2} ) \cdot \jump{ \wt{q}_h}
    \d{s} \\
    &\leq \Big(\sum_{K \in \MThG} h_K^{-2} \|\bm{v}_h\|_{L^2(K)}^2
    \Big)^{1/2} \Hcnorm{\wt{q}_h} + \sum_{f \in \MFhI} \int_f (\coed
    \bm{v}_{h, 2} ) \cdot \jump{\wt{q}_h} \d{s} \leq C \|q_h
    \|_{L^2(\Ohc)} \Hcnorm{q_h}, \\
  \end{aligned}
\end{displaymath}
and by the trace estimate, we get that
\begin{displaymath}
  \begin{aligned}
     \sum_{f \in \MFhI} \int_f &  (\coed \bm{v}_{h, 2} ) \cdot
     \jump{\wt{q}_h} \d{s} \leq C\Big( \sum_{K \in \MTh} h_K^{-2} \|
     \coed \bm{v}_{h, 2} \|_{L^2(K)}^2 + \| \nabla ( \coed \bm{v}_{h,
     2}) \|_{L^2(K)}^2 \Big)^{1/2} \Henorm{\wt{q}_h} \\ 
     & \leq C \Big(\sum_{K \in \MThG} h_K^{-2} \|\bm{v}_h\|_{L^2(K)}^2
     \Big)^{1/2} \Hcnorm{\wt{q}_h} \leq C \|\nabla \bm{v}_h
     \|_{L^2(\Oho)} \Hcnorm{\wt{q}_h} \leq \|q_h \|_{L^2(\Ohc)}
     \Hcnorm{q_h}.
  \end{aligned}
\end{displaymath}
Moreover, there hold $\|\bm{v}_{h, 1} \|_{H^1(\Ohc)} \leq C
\|\bm{v}_h \|_{H^1(\Oh)} \leq C \|q_h\|_{L^2(\Ohc)}$ and 
\begin{displaymath}
  \begin{aligned}
    (\nabla \cdot (\coed \bm{v}_{h, 1}),& \wt{q}_h)_{L^2(\Oh)} \leq
    \|\bm{v}_{h, 1} \|_{H^1(\Ohc)}  \sup_{\bm{w}_h \in \Vhcc \cap
    H_0^1(\Ohc)^d} \frac{\int_{\Ohc} \nabla \cdot (\coed \bm{w}_h)
    q_h \d{x} }{\|\bm{w}_h \|_{H^1(\Ohc)}}. \\
  \end{aligned}
\end{displaymath}
Collecting all the above estimates implies that
\begin{displaymath}
  \|q_h \|_{L^2(\Ohc)} \leq C \Big( \sup_{\bm{w}_h \in \Vhcc \cap
    H_0^1(\Ohc)^d} \frac{\int_{\Ohc} \nabla \cdot (\coed \bm{w}_h)
    q_h \d{x} }{\|\bm{w}_h \|_{H^1(\Ohc)}} + \Hcnorm{q_h} \Big),
\end{displaymath}
which, together with Lemma \ref{le_ap_infsup}, 
immediately yields the estimate \eqref{eq_infsupTH}. This completes
the proof.
\end{appendix}

\section*{Acknowledgements}
This work was supported   by National Natural Science Foundation of
China (11971041, 12171340).

\bibliographystyle{amsplain}
\bibliography{../ref}

\providecommand{\bysame}{\leavevmode\hbox to3em{\hrulefill}\thinspace}
\providecommand{\MR}{\relax\ifhmode\unskip\space\fi MR }
\providecommand{\MRhref}[2]{%
  \href{http://www.ams.org/mathscinet-getitem?mr=#1}{#2}
}
\providecommand{\href}[2]{#2}
\begin{thebibliography}{10}

\bibitem{Adams2003sobolev}
R.~A. Adams and J.~J.~F. Fournier, \emph{{Sobolev Spaces}}, second ed., Pure
  and Applied Mathematics (Amsterdam), vol. 140, Elsevier/Academic Press,
  Amsterdam, 2003.

\bibitem{Amrouche1998vector}
C.~Amrouche, C.~Bernardi, M.~Dauge, and V.~Girault, \emph{Vector potentials in
  three-dimensional non-smooth domains}, Math. Methods Appl. Sci. \textbf{21}
  (1998), no.~9, 823--864.

\bibitem{Babuska2011stable}
I.~Babu\v{s}ka and U.~Banerjee, \emph{Stable generalized finite element method
  ({SGFEM})}, Comput. Methods Appl. Mech. Engrg. \textbf{201/204} (2012),
  91--111.

\bibitem{Bermudez2002finite}
A.~Berm\'{u}dez, R.~Rodr\'{\i}guez, and P.~Salgado, \emph{A finite element
  method with {L}agrange multipliers for low-frequency harmonic {M}axwell
  equations}, SIAM J. Numer. Anal. \textbf{40} (2002), no.~5, 1823--1849.

\bibitem{Bordas2017geometrically}
S.~P.~A. Bordas, E.~Burman, M.~G. Larson, and M.~A. Olshanskii (eds.),
  \emph{Geometrically unfitted finite element methods and applications},
  Lecture Notes in Computational Science and Engineering, vol. 121, Springer,
  Cham, 2017, Held January 6--8, 2016.

\bibitem{Brenner2007locally}
S.~C. Brenner, F.~Li, and L.-Y. Sung, \emph{A locally divergence-free
  nonconforming finite element method for the time-harmonic {M}axwell
  equations}, Math. Comp. \textbf{76} (2007), no.~258, 573--595.

\bibitem{Burman2010ghost}
E.~Burman, \emph{Ghost penalty}, C. R. Math. Acad. Sci. Paris \textbf{348}
  (2010), no.~21-22, 1217--1220.

\bibitem{Burman2021unfitted}
E.~Burman, M.~Cicuttin, G.~Delay, and A.~Ern, \emph{An unfitted hybrid
  high-order method with cell agglomeration for elliptic interface problems},
  SIAM J. Sci. Comput. \textbf{43} (2021), no.~2, A859--A882.

\bibitem{Burman2015cutfem}
E.~Burman, S.~Claus, P.~Hansbo, M.~G. Larson, and A.~Massing, \emph{Cut{FEM}:
  discretizing geometry and partial differential equations}, Internat. J.
  Numer. Methods Engrg. \textbf{104} (2015), no.~7, 472--501.

\bibitem{Burman2012ficticious}
E.~Burman and P.~Hansbo, \emph{Fictitious domain finite element methods using
  cut elements: {II}. {A} stabilized {N}itsche method}, Appl. Numer. Math.
  \textbf{62} (2012), no.~4, 328--341.

\bibitem{Chen2000finite}
Z.~Chen, Q.~Du, and J.~Zou, \emph{Finite element methods with matching and
  nonmatching meshes for {M}axwell equations with discontinuous coefficients},
  SIAM J. Numer. Anal. \textbf{37} (2000), no.~5, 1542--1570.

\bibitem{Chen2007adaptive}
Z.~Chen, L.~Wang, and W.~Zheng, \emph{An adaptive multilevel method for
  time-harmonic {M}axwell equations with singularities}, SIAM J. Sci. Comput.
  \textbf{29} (2007), no.~1, 118--138.

\bibitem{Prenter2018note}
F.~de~Prenter, C.~Lehrenfeld, and A.~Massing, \emph{A note on the stability
  parameter in {N}itsche's method for unfitted boundary value problems},
  Comput. Math. Appl. \textbf{75} (2018), no.~12, 4322--4336.

\bibitem{Duan2018mixed}
H.~Duan, R.~C.~E. Tan, S.-Y. Yang, and C.-S. You, \emph{A mixed
  {$H^1$}-conforming finite element method for solving {M}axwell's equations
  with non-{$H^1$} solution}, SIAM J. Sci. Comput. \textbf{40} (2018), no.~1,
  A224--A250.

\bibitem{Ern2018analysis}
A.~Ern and J.-L. Guermond, \emph{Analysis of the edge finite element
  approximation of the {M}axwell equations with low regularity solutions},
  Comput. Math. Appl. \textbf{75} (2018), no.~3, 918--932.

\bibitem{Feng2014absolutely}
X.~Feng and H.~Wu, \emph{An absolutely stable discontinuous {G}alerkin method
  for the indefinite time-harmonic {M}axwell equations with large wave number},
  SIAM J. Numer. Anal. \textbf{52} (2014), no.~5, 2356--2380.

\bibitem{Fries2010extended}
T.-P. Fries and T.~Belytschko, \emph{The extended/generalized finite element
  method: an overview of the method and its applications}, Internat. J. Numer.
  Methods Engrg. \textbf{84} (2010), no.~3, 253--304.

\bibitem{Girault1986finite}
V.~Girault and P.-A. Raviart, \emph{Finite element methods for
  {N}avier-{S}tokes equations}, Springer Series in Computational Mathematics,
  vol.~5, Springer-Verlag, Berlin, 1986, Theory and algorithms.

\bibitem{Gurkan2019stabilized}
C.~G\"{u}rkan and A.~Massing, \emph{A stabilized cut discontinuous {G}alerkin
  framework for elliptic boundary value and interface problems}, Comput.
  Methods Appl. Mech. Engrg. \textbf{348} (2019), 466--499.

\bibitem{Massing2019stabilized}
\bysame, \emph{A stabilized cut discontinuous {G}alerkin framework for elliptic
  boundary value and interface problems}, Comput. Methods Appl. Mech. Engrg.
  \textbf{348} (2019), 466--499.

\bibitem{Guzman2018infsup}
J.~Guzm\'{a}n and M.~Olshanskii, \emph{Inf-sup stability of geometrically
  unfitted {S}tokes finite elements}, Math. Comp. \textbf{87} (2018), no.~313,
  2091--2112.

\bibitem{Xie2020extended}
Y.~Han, H.~Chen, X.~Wang, and X.~Xie, \emph{E{X}tended {HDG} methods for second
  order elliptic interface problems}, J. Sci. Comput. \textbf{84} (2020),
  no.~1, Paper No. 22, 29.

\bibitem{Hansbo2002unfittedFEM}
A.~Hansbo and P.~Hansbo, \emph{An unfitted finite element method, based on
  {N}itsche's method, for elliptic interface problems}, Comput. Methods Appl.
  Mech. Engrg. \textbf{191} (2002), no.~47-48, 5537--5552.

\bibitem{Hansbo2002discontinuous}
P.~Hansbo and M.~G. Larson, \emph{Discontinuous {G}alerkin methods for
  incompressible and nearly incompressible elasticity by {N}itsche's method},
  Comput. Methods Appl. Mech. Engrg. \textbf{191} (2002), no.~17-18,
  1895--1908.

\bibitem{Hansbo2014cut}
P.~Hansbo, M.~G. Larson, and S.~Zahedi, \emph{A cut finite element method for a
  {S}tokes interface problem}, Appl. Numer. Math. \textbf{85} (2014), 90--114.

\bibitem{Hiptmair2002finite}
R.~Hiptmair, \emph{Finite elements in computational electromagnetism}, Acta
  Numer. \textbf{11} (2002), 237--339.

\bibitem{Houston2005interior}
P.~Houston, I.~Perugia, and D.~Schneebeli, A.and~Sch\"{o}tzau, \emph{Interior
  penalty method for the indefinite time-harmonic {M}axwell equations}, Numer.
  Math. \textbf{100} (2005), no.~3, 485--518.

\bibitem{Huang2017unfitted}
P.~Huang, H.~Wu, and Y.~Xiao, \emph{An unfitted interface penalty finite
  element method for elliptic interface problems}, Comput. Methods Appl. Mech.
  Engrg. \textbf{323} (2017), 439--460.

\bibitem{Johansson2013high}
A.~Johansson and M.~G. Larson, \emph{A high order discontinuous {G}alerkin
  {N}itsche method for elliptic problems with fictitious boundary}, Numer.
  Math. \textbf{123} (2013), no.~4, 607--628.

\bibitem{Karakashian2007convergence}
O.~A. Karakashian and F.~Pascal, \emph{Convergence of adaptive discontinuous
  {G}alerkin approximations of second-order elliptic problems}, SIAM J. Numer.
  Anal. \textbf{45} (2007), no.~2, 641--665.

\bibitem{Lehrenfeld2016high}
C.~Lehrenfeld, \emph{High order unfitted finite element methods on level set
  domains using isoparametric mappings}, Comput. Methods Appl. Mech. Engrg.
  \textbf{300} (2016), 716--733.

\bibitem{Li2018interface}
R.~Li and F.~Yang, \emph{A discontinuous {G}alerkin method by patch
  reconstruction for elliptic interface problem on unfitted mesh}, SIAM J. Sci.
  Comput. \textbf{42} (2020), no.~2, A1428--A1457.

\bibitem{Li1998immersed}
Z.~Li, \emph{The immersed interface method using a finite element formulation},
  Appl. Numer. Math. \textbf{27} (1998), no.~3, 253--267.

\bibitem{Lu2017absolutely}
P.~Lu, H.~Chen, and W.~Qiu, \emph{An absolutely stable {$hp$}-{HDG} method for
  the time-harmonic {M}axwell equations with high wave number}, Math. Comp.
  \textbf{86} (2017), no.~306, 1553--1577.

\bibitem{Lu2019continuous}
P.~Lu, H.~Wu, and X.~Xu, \emph{Continuous interior penalty finite element
  methods for the time-harmonic {M}axwell equation with high wave number}, Adv.
  Comput. Math. \textbf{45} (2019), no.~5-6, 3265--3291.

\bibitem{Massing2014stabilized}
A.~Massing, M.~G. Larson, A.~Logg, and M.~E. Rognes, \emph{A stabilized
  {N}itsche fictitious domain method for the {S}tokes problem}, J. Sci. Comput.
  \textbf{61} (2014), no.~3, 604--628.

\bibitem{Monk2003finite}
P.~Monk, \emph{Finite element methods for {M}axwell's equations}, Numerical
  Mathematics and Scientific Computation, Oxford University Press, New York,
  2003.

\bibitem{Nedelec1980mixed}
J.-C. N\'{e}d\'{e}lec, \emph{Mixed finite elements in {${\bf R}^{3}$}}, Numer.
  Math. \textbf{35} (1980), no.~3, 315--341.

\bibitem{Nedelec1986new}
\bysame, \emph{A new family of mixed finite elements in {${\bf R}^3$}}, Numer.
  Math. \textbf{50} (1986), no.~1, 57--81.

\bibitem{Nguyen2011hybridizable}
N.~C. Nguyen, J.~Peraire, and B.~Cockburn, \emph{Hybridizable discontinuous
  {G}alerkin methods for the time-harmonic {M}axwell's equations}, J. Comput.
  Phys. \textbf{230} (2011), no.~19, 7151--7175.

\bibitem{Perugia2003local}
I.~Perugia and D.~Sch\"{o}tzau, \emph{The {$hp$}-local discontinuous {G}alerkin
  method for low-frequency time-harmonic {M}axwell equations}, Math. Comp.
  \textbf{72} (2003), no.~243, 1179--1214.

\bibitem{Perugia2002stabilized}
I.~Perugia, D.~Sch\"{o}tzau, and P.~Monk, \emph{Stabilized interior penalty
  methods for the time-harmonic {M}axwell equations}, Comput. Methods Appl.
  Mech. Engrg. \textbf{191} (2002), no.~41-42, 4675--4697.

\bibitem{Sarmany2010optimal}
D.~S\'{a}rm\'{a}ny, F.~Izs\'{a}k, and J.~J.~W. van~der Vegt, \emph{Optimal
  penalty parameters for symmetric discontinuous {G}alerkin discretisations of
  the time-harmonic {M}axwell equations}, J. Sci. Comput. \textbf{44} (2010),
  no.~3, 219--254.

\bibitem{Strouboulis2000design}
T.~Strouboulis, I.~Babu\v{s}ka, and K.~Copps, \emph{The design and analysis of
  the generalized finite element method}, Comput. Methods Appl. Mech. Engrg.
  \textbf{181} (2000), no.~1-3, 43--69.

\bibitem{Wei2018spatially}
Z.~Wei, C.~Li, and S.~Zhao, \emph{A spatially second order alternating
  direction implicit ({ADI}) method for solving three dimensional parabolic
  interface problems}, Comput. Math. Appl. \textbf{75} (2018), no.~6,
  2173--2192.

\bibitem{Wu2012unfitted}
H.~Wu and Y.~Xiao, \emph{An unfitted $hp$-interface penalty finite element
  method for elliptic interface problems}, J. Comput. Math. \textbf{37} (2019),
  no.~3, 316--339.

\bibitem{Yang2021unfitted}
F.~Yang and X.~Xie, \emph{An unfitted finite element method by direct extension
  for elliptic problems on domains with curved boundaries and interfaces},
  ArXiv:2112.13740 (2021), ArXiv:2112.13740.

\bibitem{Zhong2009optimal}
L.~Zhong, S.~Shu, G.~Wittum, and J.~Xu, \emph{Optimal error estimates for
  {N}edelec edge elements for time-harmonic {M}axwell's equations}, J. Comput.
  Math. \textbf{27} (2009), no.~5, 563--572.

\end{thebibliography}

\end{document}